\documentclass[
final
, nomarks
]{dmtcs-episciences}

\usepackage[utf8]{inputenc}
\usepackage{subfigure}
\usepackage{amsmath, amsthm}
\usepackage{graphicx}
\usepackage{cleveref}

\author{ L\!'ubom\'{i}ra Dvo\v{r}\'{a}kov\'{a}
  \and Kate\v{r}ina Medkov\'{a}
  \and Edita Pelantov\'{a}}

\title{Complementary symmetric Rote sequences: the critical exponent and the recurrence function\thanks{The research received funding from the Ministry of Education, Youth and Sports of the Czech Republic through the project no. CZ.02.1.01/0.0/0.0/16\_019/0000778 and from the Grant Agency of the Czech Technical University in Prague through the grant no. SGS20/183/OHK4/3T/14.}}

\affiliation{FNSPE, Czech Technical University in Prague, Czech republic}

\keywords{critical exponent, recurrence function, Rote sequence, Sturmian sequence, return word, bispecial factor}

\received{2020-3-17}
\revised{2020-5-24}
\accepted{2020-5-24}

\begin{document}

\publicationdetails{22}{2020}{1}{20}{6204}

\maketitle

\begin{abstract}
We determine the critical exponent and the recurrence function of complementary symmetric Rote sequences. The formulae are expressed in terms of the continued fraction expansions associated with the S-adic representations of the corresponding standard Sturmian sequences. The results are based on a thorough study of return words to bispecial factors of Sturmian sequences. Using the formula for the critical exponent, we describe all complementary symmetric Rote sequences with the critical exponent less than or equal to 3, and we show that there are uncountably many complementary symmetric Rote sequences with the critical exponent less than the critical exponent of the Fibonacci sequence.
Our study is motivated by a~conjecture  on sequences rich in palindromes formulated by Baranwal and Shallit.  Its recent  solution  by  Curie, Mol, and Rampersad uses two  particular complementary symmetric  Rote sequences.
\end{abstract}

\def\Q{\mathbb Q}
\def\N{\mathbb N}
\def\A{\mathcal A}
\def\S{\mathcal S}
\def\uu{\mathbf u}
\def\vv{\mathbf v}
\def\dd{\mathbf d}
\def\ind{\text{ind}}
\def\CR{\text{cr}}

\newtheorem{thm}{Theorem}
\newtheorem{coro}[thm]{Corollary}
\newtheorem{lem}[thm]{Lemma}
\newtheorem{obs}[thm]{Observation}
\newtheorem{prop}[thm]{Proposition}
\theoremstyle{definition}
\newtheorem{defi}[thm]{Definition}
\newtheorem{nota}[thm]{Notation}
\newtheorem{rem}[thm]{Remark}
\newtheorem{exam}[thm]{Example}

\section*{Introduction}

We study the relation between the critical exponents of two binary sequences $\vv = v_0v_1v_2 \cdots$ and  $\uu = u_0u_1u_2 \cdots$  over the alphabet $\{0,1\}$, where $u_i = v_i+v_{i+1} \mod 2$ for each $i \in \N$.
We write $\uu = \S(\vv)$.
Our study is motivated by a~conjecture formulated by Baranwal and Shallit in~\cite{BaSh}.
They searched for binary sequences rich in palindromes with a~minimum critical exponent.
They showed that the value of this critical exponent is greater than $2.707$.
Moreover, they found two sequences $\vv^{(1)}$ and $\vv^{(2)}$ having the critical exponent equal to $2+ \tfrac{1}{\sqrt{2}}$ and they conjectured that this is the minimum value.
Both of these sequences belong to the class of complementary symmetric Rote sequences.
Their conjecture has been recently proved by Curie, Mol, and Rampersad in \cite{CuMoRa}.

A Rote sequence is a binary sequence $\vv$ containing $2n$ factors of length $n$ for every $n \in \N, n \geq 1$.
If the language of $\vv$ is invariant under the exchange of letters $0\leftrightarrow 1$, the sequence $\vv$ is called a~complementary symmetric (CS) Rote sequence.
Already in his original paper \cite{Ro94}, Rote proved that these sequences are essentially connected with Sturmian sequences.
He deduced that a binary sequence $\vv$ is a CS Rote sequence if and only if the sequence $\uu =\S(\vv)$ is Sturmian.
Both CS Rote sequences and Sturmian sequences are rich in palindromes, see \cite{BlBrLaVu, DrJuPi}.

The formula for the critical exponent of Sturmian sequences was provided by Damanik and Lenz in \cite{DaLe}.
The relation between the critical exponent of a CS Rote sequence $\vv$ and the associated Sturmian sequence $\S(\vv)$ is not straightforward:
While the minimum exponent among all Sturmian sequences is reached by the Fibonacci sequence and it is $3+ \frac{2}{1+\sqrt{5}}$ (see \cite{MiPi}), the two CS Rote sequences $\vv^{(1)}$, $\vv^{(2)}$ whose critical exponent equals $2+ \frac{1}{\sqrt{2}}$, i.e., it is minimum among all binary rich sequences,  are associated with the Sturmian sequences $\S(\vv^{(1)})$ and $\S(\vv^{(2)})$ whose critical exponent is $3+\sqrt{2}$.

In this paper, we will first derive the relation between the critical exponents of the sequences $\vv$ and $\S(\vv)$, where $\vv$ is a ~uniformly recurrent binary sequence whose language is closed under the exchange of letters, see Theorem \ref{Lem_PrevodKritExpoRoteNaSturm}.
Using this relation, we will determine the formula for the critical exponent of any CS Rote sequence, see Theorem \ref{T_KritExpo}.

One of the consequences of this theorem is for instance the fact that the languages of the sequences $\vv^{(1)}$ and $\vv^{(2)}$  are the only languages of CS Rote sequences with the critical exponent less than $3$, see Proposition~\ref{Prop_MalyExponent}.
In this context, let us mention that in \cite{CuMoRa} the authors showed that there are exactly two languages of rich binary sequences with the critical exponent less than $\tfrac{14}{5}$ and they are the languages of the sequences $\vv^{(1)}$ and $\vv^{(2)}$.
Furthermore, we show that there are uncountably many CS Rote sequences with the critical exponent strictly less than the critical exponent of the Fibonacci sequence, see Theorem~\ref{Thm_Exponent72}.

Our main technical tool is the description of return words to bispecial factors of Sturmian sequences in terms of the continued fraction expansions related to the S-adic representations of Sturmian sequences.
As a by-product, we obtain an explicit formula for the recurrence function of CS Rote sequences, see Theorem~\ref{Thm_RecFunctionRoteDetailed}.  When formulating our results, we use the convergents  $\bigl(\tfrac{p_N}{q_N}\bigr)$  of an irrational number $\theta=[0,a_1,a_2,a_3,\ldots]$, where the coefficients $a_i$'s  in the continued fraction expansion of $\theta$ correspond to the S-adic representation of the standard Sturmian sequence associated to a given CS Rote sequence.

There are many generalizations of Sturmian sequences to multiliteral alphabets, see \cite{BaPeSt}. The critical exponent and the recurrence function were studied for two of these generalizations. Justin and Pirillo described in \cite{JuPi} the critical exponent of substitutive Arnoux-Rauzy sequences. Recently, Rampersad, Shallit, and Vandomme in  \cite{RaShVa}, and Baranwal and Shallit in \cite{BaShBalanced} determined the minimal threshold for the critical exponent of balanced sequences over alphabets of cardinality $3, 4$, and $5$, respectively.
The recurrence function of Sturmian sequences was found by Morse and Hedlund in \cite{MoHe}, and their result was generalized by Cassaigne and Chekhova in \cite{CaCh} for Arnoux-Rauzy sequences.

The paper is organized as follows.
We first introduce basic notions from combinatorics on words in Section~\ref{S_Preliminaries}. In Section~\ref{S_CriticalExponent}, we recall how to simplify the formula for the critical exponent using return words to bispecial factors.
The definitions of the already mentioned mapping $\S$ and complementary symmetric Rote sequences and their basic properties are provided in Section~\ref{S_mapS_CSRote}.
The relation between the critical exponents of the sequences $\vv$ and $\S(\vv)$ is described in Section~\ref{S_v_andSv}.
The main tool for further results -- a~thorough study of return words to bispecial factors of Sturmian sequences using the S-adic representation -- is carried out in Section~\ref{S_retwordsBS_Sturmian}.
An explicit formula for the critical exponent of CS Rote sequences is given in Section~\ref{S_CriticalExponentRote}.
CS Rote sequences with a small critical exponent are studied in Section~\ref{S_smallexponent}.
And finally, in Section~\ref{S_recurrencefunction}, an explicit formula for the recurrence function of CS Rote sequences is derived.

\section{Preliminaries} \label{S_Preliminaries}

An \textit{alphabet} $\A$ is a finite set of symbols called \textit{letters}.
A \textit{word} over $\A$ of \textit{length} $n$ is a string $u = u_0u_1 \cdots u_{n-1}$, where $u_i \in \A$ for all $i \in \{0,1, \ldots, n-1\}$. The length of $u$ is denoted by $|u|$.
The set of all finite words over $\A$ together with the operation of concatenation form a monoid $\A^*$.
Its neutral element is the \textit{empty word} $\varepsilon$ and we denote $\A^+ = \A^* \setminus \{\varepsilon\}$.

If $u = xyz$ for some $x,y,z \in \A^*$, then $x$ is a \textit{prefix} of $u$, $z$ is a \textit{suffix} of $u$ and $y$ is a \textit{factor} of $u$. We sometimes use the notation $yz = x^{-1}u$.

To any word $u$ over $\A$ with cardinality $\#\A = d$,  we assign its \textit{Parikh vector} $\vec{V}(u) \in \N^{d}$ defined as $(\vec{V}(u))_a = |u|_a$ for all $a \in \A$, where $|u|_a$ is the number of letters $a$ occurring in $u$.

A \textit{sequence} over $\A$ is an infinite string $\uu = u_0u_1u_2 \cdots$, where $u_i \in \A$ for all $i \in \N = \{0,1,2, \ldots\}$.
We always denote sequences by bold letters.
A sequence $\uu$ is \textit{eventually periodic} if $\uu = vwww \cdots = v(w)^\omega$ for some $v \in \A^*$ and $w \in \A^+$.
Otherwise $\uu$ is \textit{aperiodic}.

A \textit{factor} of $\uu$ is a word $y$ such that $y = u_iu_{i+1}u_{i+2} \cdots u_{j-1}$ for some $i, j \in \N$, $i \leq j$.
The number $i$ is called an \textit{occurrence} of the factor $y$ in $\uu$.
In particular, if $i = j$, the factor $y$ is the empty word $\varepsilon$ and any index $i$ is its occurrence.
If $i=0$, the factor $y$ is a \textit{prefix} of $\uu$.
If each factor of $\uu$ has infinitely many occurrences in $\uu$, the sequence $\uu$ is \textit{recurrent}.
Moreover, if for each factor the distances between its consecutive occurrences are bounded, $\uu$ is \textit{uniformly recurrent}.

The \textit{language} $\mathcal{L}(\uu)$ of the sequence $\uu$ is the set of all factors of $\uu$.
A factor $w$ of $\uu$ is \textit{right special} if both words $wa$ and $wb$ are factors of $\uu$ for at least two distinct letters $a,b \in \A$.
Analogously we define a \textit{left special} factor.
A factor is \textit{bispecial} if it is both left and right special.
Note that the empty word $\varepsilon$ is a bispecial factor if at least two distinct letters occur in $\uu$.

The \textit{factor complexity} of a sequence $\uu$ is a mapping $\mathcal{C}_\uu: \N \to \N$ defined by
$$\mathcal{C}_\uu(n) = \# \{w \in \mathcal{L}(\uu) : |w| =  n \}\, .$$
The aperiodic sequences with the lowest possible factor complexity are called \textit{Sturmian sequences}.
In other words, it means that a sequence $\uu$ is \textit{Sturmian} if it has the factor complexity $\mathcal{C}_\uu(n) = n+1$ for all $n \in \N$.
Clearly, all Sturmian sequences are defined over a binary alphabet, e.g., $\{0, 1\}$.
There are many equivalent definitions of Sturmian sequences, see a survey in \cite{BaPeSt}.

A \textit{morphism} over $\A$ is a mapping $\psi: \A^* \to \A^*$ such that $\psi(uv) = \psi(u)\psi(v)$  for all $u, v \in \A^*$.
The morphism $\psi$ can be naturally extended to sequences by
$$\psi(\uu)=\psi(u_0u_1u_2 \cdots) = \psi(u_0)\psi(u_1)\psi(u_2) \cdots\,. $$
A \textit{fixed point} of a morphism $\psi$ is a sequence $\uu$ such that $\psi(\uu) = \uu$.
The \textit{matrix} of a morphism $\psi$ over $\A$ with the cardinality $\#\A = d$ is the matrix $M_\psi \in \N^{d\times d}$ defined as $(M_\psi)_{ab} = |\psi(a)|_b$ for all $a,b \in \A$.
The Parikh vector of the $\psi$-image of a word $w\in \A^*$ can be obtained via multiplication by the matrix $M_\psi$, i.e.,
\begin{equation}\label{multi}\vec{V}({\psi(w)}) = M_\psi \vec{V}(w)\,.
\end{equation}

Consider a prefix $w$ of a recurrent sequence $\uu$.
Let $i < j$ be two consecutive occurrences of $w$ in $\uu$.
Then the word $u_iu_{i+1} \cdots u_{j-1}$ is a \textit{return word} to $w$ in $\uu$.
The set of all return words to $w$ in $\uu$ is denoted by $\mathcal{R}_\uu(w)$.
If the sequence $\uu$ is uniformly recurrent, the set $\mathcal{R}_\uu(w)$ is finite for each prefix $w$, i.e., $\mathcal{R}_\uu(w) = \{r_0, r_1, \ldots, r_{k-1}\}$.
Then the sequence $\uu$ can be written as a concatenation of these return words:
$$\uu = r_{d_0}r_{d_1}r_{d_2} \cdots$$
and the \textit{derived sequence} of $\uu$ to the prefix $w$ is the sequence $\dd_\uu(w) = d_0d_1d_2 \cdots$ over the alphabet of cardinality $\# \mathcal{R}_\uu(w) = k$.
The concept of derived sequences was introduced by Durand in  \cite{Dur98}.

\section{The critical exponent and its  relation to return words} \label{S_CriticalExponent}
Let $z \in \mathcal{A}^+$ be a prefix of a periodic sequence $u^\omega$ with $ u \in \mathcal{A}^+$.
We say that $z$ has the \emph{fractional root} $u$ and the \emph{exponent} $e = |z|/|u|$.
We usually write $z =u^e$.
Let us emphasize that a word $z$ can have multiple exponents and fractional roots.
A word $z$ is \emph{primitive} if its only integer exponent is $1$.

Let $\uu$ be a sequence and $u$ its non-empty factor.
The supremum of $e \in \Q$ such that $u^e$ is a factor of $\uu$ is the \emph{index} of $u$ in $\uu$:
\begin{align*}
\ind_\uu(u) = \sup\{e \in \Q: u^e \in \mathcal{L}(\uu)\}\,.
\end{align*}
If the sequence $\uu$ is clear from the context, we will write $\ind(u)$ instead of $\ind_\uu(u)$.

\begin{defi}
The \emph{critical exponent} of a sequence $\uu$ is
\begin{align*}
\CR(\uu) &= \sup \, \{e \in \Q : \text{ there is a non-empty factor of  } \uu \text{ with the exponent } e\} \\
       &= \sup \, \{ \ind_\uu(u) : u \text{ is a non-empty factor of } \uu\}\,.
\end{align*}
\end{defi}

\begin{rem}\label{remark}  Let us comment on the above definition.
\begin{enumerate}
\item If a non-empty factor $u \in \mathcal{L}(\uu)$ is non-primitive, i.e.,  $u=x^k$ for some $x \in \A^+$ and $k \in \mathbb{N}, k\geq 2$, then $\ind_\uu(x) = k\, \ind_\uu(u) > \ind_\uu(u)$. Therefore,  only primitive factors play a role for finding $\CR(\uu)$.
\item If some non-empty factor occurs at least twice in $\uu$, then $\ind(x)> 1$ for some non-empty factor $x$ and so $ \CR(\uu) >1$.
Consequently,  $ \CR(\uu)  >1$  for each sequence $\uu$.
\item We say that $u$ is an \emph{overlapping} factor in $\uu$, if there exist $x, y \in \mathcal{A}^*$ such that $xu=uy\in \mathcal{L}(\uu)$ and $ 0< |x|  <|u|$.
If $\uu$ has an overlapping factor, then $\CR(\uu) > 2$.
Indeed, by \cite{Lo83} the equality $xu=uy$ implies that there exist $a,b \in \mathcal{A}^*$ and $k \in \mathbb{N}$ such that $u=(ab)^ka$, $x=ab$, and $y=ba$.  If $a$ is empty then the assumption $|u| >|x|>0$ forces $k\geq 2$,  otherwise $k \geq 1$.  In both cases  $ \ind_\uu(ab)>2$.
\item If $\uu$ is eventually periodic, then $\CR(\uu)$ is infinite.
\item If $\uu$ is aperiodic and uniformly recurrent, then each factor of $\uu$ has a finite index.
Nevertheless, $\CR(\uu)$ may be  infinite.
A Sturmian sequence may serve as an example of such a sequence if the continued fraction expansion of its slope has unbounded partial quotients, see \cite{DaLe}.
\item If $\uu$ is a binary sequence, then either $11$, $00$, or $0101$ occur in $\uu$.
It means that the critical exponent of a binary sequence is at least $2$.
This value is attained by the famous Thue-Morse sequence, which is,  of course, overlap-free, see \cite{Thue} or \cite{Berstel}.
The critical exponent  of generalized Thue-Morse sequences over multi-letter alphabets were described in \cite{BlBrGlLa}.
\item Recently, Ghareghani and Sharifani introduced a natural generalization of $k$-bonacci sequences to infinite alphabets. They also computed the critical exponent of these sequences, \cite{GhSh}.
\end{enumerate}
\end{rem}

\begin{lem} \label{Lem_CriticalExponentByReturnWords}
Let $\uu$ be a uniformly recurrent  aperiodic sequence.
Then $\emph{cr}(\uu) = \sup \, \{ \emph{ind}_\uu(u) : u \in \mathcal{M}\}$, where
\begin{align*}
\mathcal{M} = \{ u : u  \text{ is a return word to a bispecial factor of } \uu \}\,.
\end{align*}
\end{lem}

\begin{proof}
Let $u \in \mathcal{L}(\uu)$ be a non-empty factor with the index $\ind(u) >1$.
Denote $|u| = n$.
When searching for supremum, we may assume without loss of generality that $u$ is a factor having the largest index among all factors of $\uu$ of length $n$, i.e., $\ind(u) \geq \ind(v)$ for all $v \in \mathcal{L}(\uu)$ of length $n$.
Since $\uu$ is uniformly recurrent, $\ind(u)$ is finite.
We denote
\begin{align*}
z = u^{\ind(u)} = u'u'' \cdots u'u''u' \quad \text{ and } \quad b = u^{\ind(u)-1}\,,
\end{align*}
where $u = u'u''$ and $u'' \neq \varepsilon$.
Clearly, $z = bu''u' = u'u''b$.

Let us show that the word $b$ is a bispecial factor of $\uu$.
The word $z$ is a factor of $\uu$ and so $z$ occurs in $\uu$ at some position $j$, i.e., $$z = u_ju_{j+1} \cdots u_{j + |z| -1}\,.$$
Then the letter $u_{j + |z|}$ which follows the word $z$ is distinct from the first letter of $u''$. Otherwise, we could prolong $z$ to the right, which contradicts the definition of the index of $u = u_ju_{j+1} \cdots u_{j+n-1}$.
Similarly, the letter $u_{j-1}$ which precedes $z$ is distinct from the last letter of $u''$.
Indeed, if those letters are the same, then the factor $u_{j-1}u_{j} \cdots u_{j + n-2}$ of length $n$ has the index at least $\ind(u) + \frac{1}{n}$, which contradicts the choice of $u$.
We can conclude that the factor $b$ is a bispecial factor of $\uu$.

Moreover, since $z = bu''u' = u'u''b$, the word $u = u'u''$ is a concatenation of the return words to the bispecial factor $b$.
It suffices to prove that only the cases when $u$ is a  return word to the bispecial factor $b$ have to be inspected.
Let us assume  that  $u$ is a concatenation of at least two return words to $b$.
It means that
\begin{equation} \label{Eq_Rovnice}
z = ub = sbt  \quad \text{ for some } s, t \text{ such that $s$ is a prefix of } u \text{ and } 0 < |s| < |u|\,.
\end{equation}
We will find another factor of $\uu$ with the index strictly larger than $\ind(u)$, which means that such a factor $u$ can be omitted.
We distinguish three cases:
\begin{itemize}
\item If $\ind(u) \geq 2$, then $|b| \geq |u|$ and both words $u$ and $s$ are prefixes of $b$.
Therefore, the relation \eqref{Eq_Rovnice} implies $us = su$ and we can easily conclude that there is a word $x$ and an integer $k > 1$ such that $u = x^k$.
As mentioned in Remark \ref{remark}, $\ind(x)> \ind(u)$.
\item If $1 < \ind(u) <2 $ and $\CR(\uu) \leq 2$, then by Item 3 of Remark \ref{remark}, $\uu$ has no overlapping factor.
Clearly, $z=u'u''u'$ and $b=u'$ for $u', u'' \neq \varepsilon$.
Then the relation \eqref{Eq_Rovnice} implies $u'v = su'$ for some $v$ and $|u| > |s| \geq |u'|$.
Indeed, if $0 < |s| < |u'|$, then $u'$ is an overlapping factor, which is not possible.
Therefore, $u'$ is a prefix of $s$ and we can easily deduce that
$$
\ind(s) \geq  \frac{|s|+  |u'| }{|s|} > \frac{|u|+ |u'| }{|u|} = \ind(u) \,.
$$
\item If $1 < \ind(u) < 2$ and $\CR(\uu) > 2$, then there is a factor $x \in \mathcal{L}(\uu)$ with $\ind(x) > 2 > \ind(u)$.
\end{itemize}
\end{proof}

\begin{rem} \label{rem_CriticalExponentByReturnWords}
In fact, we proved that it suffices to consider the set
\begin{align*}
\mathcal{M'} = \{ u : u  \text{ is a return word to a bispecial factor of } \uu \text{ with the fractional root } u \}\,
\end{align*}
or, even more specifically, the set
\begin{align*}
\mathcal{M''} = \{ u : u  \text{ is a return word to the bispecial factor } b = u^{\ind(u)-1} \text{ of } \uu \}\,
\end{align*}
instead of $\mathcal{M}$.
Clearly, $\mathcal{M''} \subset \mathcal{M'} \subset \mathcal{M}$.

In Remark \ref{remark}, we emphasize that only primitive factors are relevant for finding the critical exponent.
Let us verify that all return words from the set $\mathcal{M'}$ (and so $\mathcal{M}''$, too) are primitive.
We prove it by contradiction.
Let us suppose that $u \in \mathcal{M'}$ is non-primitive, i.e., $u = x^k$ for some non-empty $x$ and $k \in \N$, $k > 1$.
Since $b$ has the fractional root $u$, $b = u^\ell = x^{k \ell}$ for some $\ell \in \Q$.
Therefore, $ub = x^{k(\ell +1)} = xx^{k\ell}x^{k-1} = xbx^{k-1}$, which contradicts that $u$ is a return word to $b$ in $\uu$.
\end{rem}

\section{The mapping $\mathcal{S}$ on binary words and complementary symmetric Rote sequences}\label{S_mapS_CSRote}

In this section, we introduce a mapping $\S$ which enables us to describe the properties of CS Rote sequences using Sturmian sequences. Nevertheless, this mapping $\S$ can be applied to any binary sequence.

\begin{defi}\label{S}
By $\S$ we denote the mapping $\S: \{0,1\}^+ \mapsto \{0,1\}^*$ such that for every $v_0\in \{0,1\}$ we put $\S(v_0) = \varepsilon$ and for every $v = v_0v_1 \cdots v_{n} \in \{0,1\}^+$ of length at least $2$ we put $\S(v_0v_1 \cdots v_{n}) = u_0u_1 \cdots u_{n-1}$, where
$$
u_{i} = v_i + v_{i+1} \mod 2 \  \ \text{ for all } \ i \in \{0,1, \ldots, n-1\}\,.
$$
Moreover, we extend the domain of $\S$ naturally to $\{0,1\}^\N$: for every $\vv \in \{0,1\}^\N$ we put $\S(\vv) = \uu$, where
$$
u_{i} = v_i + v_{i+1} \mod 2 \  \ \text{ for all } \ i \in \N \,.
$$
\end{defi}

By $E:\{0,1\}^*\mapsto \{0,1\}^*$ we denote the morphism which exchanges the letters, i.e., $E(0)=1$, $E(1)=0$.

\begin{exam}
We have $E(001110) = 110001$ and $\S(001110) = \S(110001)  = 01001$.

\end{exam}

Clearly, the images of $v$ and $E(v)$ under $\S$ coincide for each $v \in \{0,1\}^*$.
Moreover, $\S(x) = \S(y)$ if and only if $x = y$ or $x = E(y)$.
The following rule follows directly from the definition of $\S$:
\begin{equation}\label{jasne}
\S(v_0v_1 \cdots v_{n}) = \S(v_0 v_1\cdots v_k) \S(v_k v_{k+1} \cdots v_{n})\quad \text{  for any $k = 0, \ldots, n$}.
\end{equation}
These observations hold also for infinite sequences.

\begin{lem} \label{lem_Bispecial}
Let $\vv$ be a binary sequence whose language $\mathcal{L}(\vv)$ is closed under $E$.
Then $w\neq \varepsilon$ is a right (left) special factor in $\vv$  if and only if $\S(w)$ is a right (left) special factor in $\S(\vv)$.
\end{lem}
\begin{proof}
We will prove the statement for right special factors. The proof for left special factors is analogous. Let  $c$  be the last letter  of $w$.

\noindent $(\Longrightarrow)$: Let $w0$ and $w1$ belong to $\mathcal{L}(\vv)$. Then $\S(w0)$ and $\S(w1)$ belong to $\mathcal{L}(\S(\vv))$.
By the rule \eqref{jasne}, $\S(w0) =\S(w)\S(c0)$ and $\S(w1) =\S(w)\S(c1)$.
As $\S(c0) \neq \S(c1)$, the factor $\S(w)$ is right special in $\S(\vv)$.

\noindent $(\Longleftarrow)$: Let $\S(w)0$  and $\S(w)1$ be in  $\mathcal{L}(\S(\vv))$.
Then $\S(w)0 = \S(w)\S(cc) = \S(wc)$ and $\S(w)1 = \S(w)\S(cE(c)) = \S(wE(c))$.
Since  $\mathcal{L}(\vv)$ is closed under the exchange of letters, all factors  $wc$, $E(wc)$, $wE{(c)}$ and $E(w)c$ are in $\mathcal{L}(\vv)$.
It means that $w$ and $E(w)$ are right special factors in $\vv$.

\end{proof}

A \textit{Rote sequence} is a sequence $\vv$ with the factor complexity $\mathcal{C}_\vv(n) = 2n$ for all $n \in \N, n \geq 1$.
Clearly, all Rote sequences are defined over a binary alphabet, e.g., $\{0,1\}$.
If the language of a Rote sequence $\vv$ is  closed under the exchange of letters, i.e., $E(v)\in \mathcal{L}(\vv)$ for each $v \in \mathcal{L}(\vv)$, the Rote sequence  $\vv$  is called \textit{complementary symmetric} (shortly CS).

Rote in \cite{Ro94} proved that these sequences are essentially connected with Sturmian sequences.

\begin{prop}[\cite{Ro94}] \label{Prop_SturmRote}
Let $\uu$ and $\vv$ be two sequences over $\{0,1\}$ such that $\uu=\S(\vv)$. Then $\vv$ is a~complementary symmetric Rote sequence if and only if $\uu$ is a Sturmian sequence.
\end{prop}

Let us emphasize that to a given CS Rote sequence $\vv$ there is the unique \emph{associated} Sturmian sequence $\uu$ such that $\uu = \S(\vv)$.
On the other hand, for any Sturmian sequence $\uu$ there exist two \emph{associated} CS Rote sequences $\vv$ and $E(\vv)$ such that  $\uu = \S(\vv)  = \S(E(\vv))$.
However, $\mathcal{L}(\vv) = \mathcal{L}(E(\vv))$.

Analogously, to a given factor $v \in \mathcal{L}(\vv)$ there is a unique associated word $u$ such that $u = \S(v)$ and this word $u$ is a factor of $\uu$.
In addition, to a given factor $u \in \mathcal{L}(\uu)$ there are exactly two associated words $v, E(v)$ such that $\S(v) = \S(E(v)) = u$ and both these words $v, E(v)$ are factors of $\vv$.

\begin{exam}\label{Fib}
Let us underline that for Sturmian sequences $\uu$ and $E(\uu)$ the languages of their associated CS Rote sequences may essentially differ.
Consider the Fibonacci sequence
$${\mathbf f}=abaababaaba\cdots\,,$$ which is the fixed point of the Fibonacci morphism $F: a \to ab$, $b \to a$.

\begin{itemize}
\item
If $a=0$ and $b=1$, then the associated CS Rote sequence starting with $0$ is $\vv=001110011100\cdots$.
The prefix $w$ of $\vv$ of length 7 is $w=0011100=(00111)^{\frac{7}{5}}$, i.e., $w$ has the fractional root $00111$ and the exponent $\frac{7}{5}$.
\item
If $a=1$ and $b=0$, then the associated CS Rote sequence starting with $0$ is $\vv'=011011001001\cdots$.
 The prefix $w'$ of $\vv'$ of length 7 is $w'=0110110=(011)^{\frac{7}{3}}$, i.e., $w'$ has the fractional root $011$ and the exponent $\frac{7}{3}$.
\end{itemize}
We will show later in Example~\ref{Ex_Rozdil_cr_u_Eu} that even the critical exponent of CS Rote sequences associated with $\uu$ and $E(\uu)$ may be different.
\end{exam}

In the next section, we will explain that the relation between the shortest fractional root of a factor $v$ and the shortest fractional root of $\S(v)$ is influenced by the number of letters $1$ occurring in the shortest fractional root of $\S(v)$. This is the reason for the following definition and lemma.

\begin{defi}
A word $u = u_0u_1 \cdots u_{n-1} \in \{0, 1\}^*$ is called \emph{stable} if  $|u|_1 = 0 \mod 2$. Otherwise, $u$ is \emph{unstable}.
\end{defi}

\begin{lem}\label{pomocny}  Let $\S: \{0,1\}^+\mapsto  \{0,1\}^*$.
\begin{itemize}
\item[\textit{(i)}] If $0$ is a prefix of $v \in \{0,1\}^*$, then $\S(v0)$ is stable.
\item[\textit{(ii)}] For every $u \in \{0,1\}^*$ there exists a unique $w \in \{0,1\}^+$ with a prefix $0$ such that $u = \S(w)$.  Moreover, $w$ has a suffix $0$ if and only if $u$ is stable.
\item[\textit{(iii)}] If $0$ is a prefix of $w$, then $\S(vw) = \S(v0)\S(w)$.
\item[\textit{(iv)}] Let $0$ be a prefix of $v$ and $v'$.
Then $\S(v')$ is a prefix of $\S(v)$ if and only if $v'$ is a prefix of $v$.
\end{itemize}
\end{lem}

\begin{proof}
\begin{itemize}
\item[(i)]  Let $v = v_0v_1\cdots v_{n-1}$, $n =|v|$, and $v_0=0$.
Put $v_n =0$. Then $\S(v0) = u_0u_1\cdots u_{n-1}$, where $u_i =v_{i} + v_{i+1} \mod 2 $ for every $ i\in \{0,1, \ldots, n-1\}$.
It implies
$$|\S(v0)|_1 = \sum_{i=0}^{n-1}u_i = v_0 + v_n = 0 \mod 2\,.$$
\item[(ii)]  Let $u= u_0u_1\cdots u_{m-1}$ and $m =|u|$.
We look for $w =w_0 w_1\cdots w_m$ such that $u_i = w_i+w_{i+1} \mod 2$ for every $i\in \{0,1,\ldots, m-1\}$.
Clearly, these equations can be equivalently rewritten as
\begin{equation} \label{Eq_Kongruence}
u_i = w_{i+1} - w_{i} \mod 2 \quad \text{ for every } i\in \{0,1,\ldots, m-1\}\,.
\end{equation}
Then starting with $w_0=0$ and summing up the equations \eqref{Eq_Kongruence}  for $i=\{0, 1, \ldots, j-1\}$, we determine the letter $w_j$ of $w$ as $w_{j} = \sum_{i=0}^{j-1} u_i \mod 2$.
In particular, $w_m = |u|_1\mod 2$.
\item[(iii)] It is a particular case of the equation \eqref{jasne}.
\item[(iv)] It follows directly from  the definition of $\S$.
\end{itemize}
\end{proof}

\section{The relation between the indices of factors in $\vv$ and $\S(\vv)$} \label{S_v_andSv}

In this section, we provide a tool for determining  the critical exponent of a binary sequence $\vv$  whose language is closed under the exchange of letters.
For any factor $v$  of such a sequence, $\ind_\vv(v) =\ind_\vv(E(v))$ and we can consider only factors of $\vv$ starting with $0$ without loss of generality.

\begin{lem}\label{prevod}
Let $\vv$ be a binary aperiodic uniformly recurrent sequence whose language is closed under $E$.
Denote $\uu = \S(\vv)$.
For a non-empty factor $v \in \mathcal{L}(\vv)$ with the prefix $0$ and $\emph{ind}_\vv(v) > 1$, there exists a~stable factor $u \in \mathcal{L}(\uu)$  such that
\begin{equation}\label{vztah}
\emph{ind}_\uu(u) +  \tfrac{1}{|u|} = \emph{ind}_\vv(v)  \quad  \text{and} \quad u= \S(v0)\,.
\end{equation}
And  vice versa, for a non-empty stable factor $u \in \mathcal{L}(\uu)$, there exists a factor $v \in \mathcal{L}(\vv)$ with the prefix $0$ satisfying
\eqref{vztah}\,.
\end{lem}

\begin{proof}
For a given $n \in \mathbb{N}, n\geq 1$, consider the set $\mathcal{K}_n$  of factors  $v \in \mathcal{L}(\vv)$ of length $n$ with the prefix $0$ and $\ind_\vv(v) > 1$.
First, we show that the mapping $v \mapsto \S(v0)$ is a bijection between $\mathcal{K}_n$ and the set of all stable factors of $\uu$ of length $n$.

Indeed, if  $v \in \mathcal{K}_n$, then $v0  \in \mathcal{L}(\vv)$.
The factor $u: = \S(v0)$ belongs to $ \mathcal{L}(\uu)$, $|u| = |v|$, and by Item (i) of  Lemma \ref{pomocny}, $u$ is stable.
On the other hand, if $u \in \mathcal{L}(\uu)$ is stable and of length $n$,  then by Item (ii) of Lemma \ref{pomocny}, there exists a unique $w$ such that $0$ is a prefix and a suffix of $w$ and $\S(w) = u$.
As $\mathcal{L}(\vv)$ is closed under $E$,  necessarily $w \in \mathcal{L}(\vv)$ and $w =v0$  for some $v$ with the prefix $0$.
In particular, $\ind_\vv(v) > 1$. As $u =\S(w) =\S(v0)$, the lengths of $u$ and $v$ coincide.
\medskip

Now we show that any $v  \in  \mathcal{K}_n$ and its image $u =\S(v0)$ satisfy  \eqref{vztah}.
Find $k \in \mathbb{N}, k\geq 1$, and $\theta \in (0, 1]$ such that  $\ind_\vv(v) =k + \theta$.
Denote $v' = v^\theta$.
Obviously, $v' \neq \varepsilon$, $v'$ is a prefix of $v$ and $0$ is a prefix of $v'$.
Applying Item (iii) of Lemma \ref{pomocny}, we get $\S(v^kv')= (\S(v0))^k \S(v')$.
Clearly, $|\S(v')| = |v'| -1$, $|u| = |v|$, and by Item (iv) of Lemma \ref{pomocny}, $\S(v')$ is a prefix of $\S(v0)$.
For $u = \S(v0)$ it means that
$$\ind_\uu(u) \geq  k + \tfrac{|S(v')|}{|u|} = k + \tfrac{|v'|}{|v|} - \tfrac{1}{|u|} = k + \theta - \tfrac{1}{|u|} = \ind_\vv(v)  - \tfrac{1}{|u|}\,.$$

To show the opposite inequality, we find $\ell \in \mathbb{N}$  and $\eta \in [0,1)$  such that   $\ind_\uu(u) =\ell + \eta$.
Denote $u' = u^\eta$.
Using Item (ii) of Lemma \ref{pomocny},  we find $v'$ with the prefix $0$ such that $u' = \S(v')$.
By Item (iv), $v'$ is a prefix of $v$, and by Item (iii), $u^\ell u' = (\S(v0))^\ell \S(v') = \S(v^\ell v')$.
Therefore, $v^\ell v' \in \mathcal{L}(\vv)$  and $$\ind_\vv(v)  \geq \ell+\eta + \tfrac{1}{|u|} = \ind_\uu(u)+ \tfrac{1}{|u|}\,.$$
\end{proof}

As explained in Lemma \ref{Lem_CriticalExponentByReturnWords}, only return words to bispecial factors play a role for the determination of the critical exponent of a sequence.
More specifically, we can restrict ourselves to factors from the set $\mathcal{M}'$ (or $\mathcal{M}''$) introduced in Remark \ref{rem_CriticalExponentByReturnWords}.

\begin{lem} \label{CoJeToProU}
Let $\vv$ be a binary sequence whose language is closed under $E$.
Assume that $v$ with the prefix $0$ is a return word in $\vv$ to a bispecial factor $b = v^{e-1}$, where  $e> 2$.
Denote $u =\S(v0)$.
Then

--  either $u$ is a stable return word in $\S(\vv)$ to a bispecial factor with the fractional root $u$;

-- or  $u = x^2$, where $x$ is an unstable return word in $\S(\vv)$ to a bispecial factor with the fractional root~$u$.
\end{lem}

\begin{proof}
The factor $\S(b)$ is bispecial in $\S(\vv)$ by Lemma \ref{lem_Bispecial}.
Moreover, by the rule \eqref{jasne}, we can write $\S(b) = \S(v^{e-1}) = \S(v0)^{f}$ for $f = e-1-\frac{1}{|v|} \geq 1$. Thus $\S(b)$ has the fractional root $u = \S(v0)$.

The word $vb$ is a complete return word to $b$ in $\vv$ and thus $vb =bw$ for some $w$.
Note that $0$ is the first letter of $b$. We denote the last letter of $b$ by $z$.  By the rule \eqref{jasne}, we get
$\S(v0)\S(b) = \S(b)\S(zw)$. It means that $\S(b)$ is a prefix and a suffix of the word $\S(v0)\S(b)$.
We discuss two cases:

--  $\S(b)$  has exactly two occurrences in $\S(v0)\S(b)$, one as a prefix   and one  as a suffix. In this case  $u= \S(v0)$ is a return word to $\S(b)$ in $\S(\vv)$  and by Item (i) of Lemma \ref{pomocny},  $u$ is stable.

\medskip
-- $\S(b)$  occurs in  $\S(v0)\S(b)$ as an inner factor. In this case, there exists a return word   $u'\neq \varepsilon$ to $\S(b )$  such that  $|u'| <|u| $ and $u'\S(b)$ is a proper prefix of $\S(v0)\S(b)$.
We take the word $b'$ such that $b= v b'$. Clearly, $b'$ has the prefix $0$ and so $\S(b) = u\S(b')$.
Then $u'\S(b) = u'u\S(b')$ is a proper prefix of $\S(v0)\S(b) = u\S(b) = uu\S(b')$.
In other words, $u'uu'' = uu$ for some non-empty $u''$ and consequently,
$u =u'u'' = u''u'$.
This implies the existence of $x \in \{0,1\}^+$ and $k', k'' \in \mathbb{N}, k', k''\geq 1$ such that $u' =x^{k'}$  and $u'' = x^{k''}$.
If we denote  $k =k'+k'' \geq 2$, we can write $u=x^k$.

We show that $x$ is unstable and $k=2$.
Indeed, assume $x$ is stable, then by Item (ii) of Lemma \ref{pomocny}, we find a unique $y$ with the prefix $0$ such that $ x =\S(y0)$.
Applying  Item (iii), we obtain $\S(v0)=u= x^k = (\S(y0))^k = \S(y^k0)$ and thus $v =y^k$.
Nevertheless, the factor $v$ is primitive as explained in Remark \ref{rem_CriticalExponentByReturnWords}. Thus this is a contradiction.

Since $x$ is unstable and $u=x^k$ is stable, necessarily  $k = 2p$ for some integer $p\geq 1$.
Now we deduce that $k=2$. Indeed, if  $p\geq 2$, then  $u$ is a $p$-power of the stable factor  $x^2$ which yields a contradiction with the primitivity of $v$ as above.
Finally, $k=2$ implies $k' = 1$ and $u' = x$ is an unstable return word to the bispecial factor $\S(b)$ in $\S(\vv)$.
\end{proof}

\begin{thm} \label{Lem_PrevodKritExpoRoteNaSturm}
Let $\vv$ be a binary aperiodic uniformly recurrent sequence whose language is closed under $E$.
Denote  $\uu = \S(\vv)$,

$A_1 =  \left\{ \emph{ind}_\uu(u) + \tfrac{1}{|u|}  : u \text{ is a stable  return word to a bispecial factor of } \uu \right\} \ $ and

$A_2 =  \left\{\tfrac12 \bigl(\emph{ind}_\uu(u) + \tfrac{1}{|u|} \bigr)  : u \text{ is an unstable  return word to a bispecial factor of } \uu \right\}  \,.$
\medskip

\noindent Then
$$ \emph{cr}(\vv) = \sup  \bigl(A_1\cup A_2\bigr)\,.$$
\end{thm}

\begin{proof}
First we show that
\begin{equation}\label{mensi}
\ind_{\vv}(v)\leq \sup  \bigl( A_1\cup A_2\bigr) \quad \text{  for any non-empty} \ v \in \mathcal{L}(\vv)\,.
\end{equation}

If $\ind_{\vv}(v)\leq 2$, then the inequality \eqref{mensi} is trivially satisfied as $A_1$ contains the number $\ind_{\uu}(0) + \tfrac{1}{|0|} \geq 2$  (note that $0$ is a  stable  return word  in $\uu$ to the bispecial factor $\varepsilon$).
Now we assume that $\ind_{\vv}(v) = e> 2$.
By Lemma \ref{Lem_CriticalExponentByReturnWords} and Remark \ref{rem_CriticalExponentByReturnWords}, we may focus only on $v$ which is a  return word to the bispecial factor $b =v^{e-1}$ and $v$ has the prefix $0$.
By Lemma \ref{prevod}, $\ind_{\vv}(v) = \ind_{\uu}(u) + \tfrac{1}{|u|}$, where $u= \S(v0)$.
By Lemma \ref{CoJeToProU}, the factor $u$ is either a stable return word to a bispecial factor in $\uu$, or $u = x^2$, where $x$ is an unstable return word to a bispecial factor in $\uu$.
In the first case we have $\ind_{\vv}(v) \leq \sup A_1$, while in the second case we have $\ind_{\vv}(v) = \ind_{\uu}(u) + \tfrac{1}{|u|} = \tfrac{1}{2}\ind_{\uu}(x) + \tfrac{1}{2|x|} \leq \sup A_2$.
We may conclude that $ \CR(\vv) \leq  \sup  \bigl(A_1\cup A_2\bigr)\,.$

To prove the opposite inequality, we show
$$  A_1\cup A_2 \setminus \left[0,1\right] \subset \{\ind_\vv(v) : v \in \mathcal{L}(\vv), v\neq \varepsilon\}\,.$$
If $H \in A_1$, then there exists a stable factor $u$ in $\uu$ such that
$H= \ind_\uu(u) + \frac{1}{|u|}$, and by Lemma \ref{prevod}, we find  $v$ in $\vv$ such that $H=\ind_\vv(v)$.
Analogously, if  $H \in A_2$ and $H >1$, then $\ind_\uu(u)  = 2H - \frac{1}{|u|} \geq 2$ for some unstable factor $u  \in \mathcal{L}(\uu)$.
Thus the word $y=uu \in \mathcal{L}(\uu)$, it is a stable factor of $\uu$ and its index in $\uu$ is  $\tfrac12\ind_\uu(u)$.
By Lemma \ref{prevod}, there exists $v$ in $\vv$ such that
$\ind_\vv(v)  = \ind_\uu(y) + \tfrac{1}{|y|} =  \tfrac12\ind_\uu(u) + \tfrac{1}{2|u|} = H.$
\end{proof}

Theorem \ref{Lem_PrevodKritExpoRoteNaSturm} will be used in the next sections to determine the critical exponent of a complementary symmetric Rote sequence $\vv$ by exploiting the indices of factors in the Sturmian sequence $\S(\vv)$.
The following example shows an opposite application of Theorem \ref{Lem_PrevodKritExpoRoteNaSturm}.
But before that, let us state a simple auxiliary statement reflecting the behaviour of fractional roots under the application of a morphism.

\begin{obs}\label{pridano}
Let $\phi :\mathcal{A}^* \mapsto \mathcal{A}^* $ be a morphism and let $w \in \mathcal{A}^*$ be a prefix of $\phi(a)$ for each $a\in \mathcal{A}$.
If $u$ is a fractional root of $z$, then $\phi(u)$ is a fractional root of $\phi(z)w$.
\end{obs}

\begin{exam}
Let us consider the Thue--Morse sequence  $\mathbf{t} = 01101001\cdots$, which is fixed by the morphism $\psi:  0\mapsto 01$ and $1\mapsto 10$.
As $\psi$ is primitive, the sequence $\mathbf{t}$ is uniformly recurrent.
It is well-known that its language is closed under the exchange of letters and $\CR(\mathbf{t}) = 2$.
The corresponding sequence $\uu = \S(\mathbf{t}) = 1011101 \cdots$ is called the  period doubling sequence and it is fixed by the morphism $\phi : 0\mapsto 11$ and $1\mapsto 10$, see \cite{rigo}.

We determine the critical exponent of $\uu$.
Theorem \ref{Lem_PrevodKritExpoRoteNaSturm} implies $\CR(\uu)\leq 4$, as otherwise $\CR(\mathbf{t}) >2$, which is a contradiction.
Now we show that the value $4$ is attained.

By Observation \ref{pridano} and the fact that both $\phi(0)$ and $\phi(1)$ have the prefix $1$, the morphism $\phi$ has the following two properties:
\begin{enumerate}
\item If $w \in \mathcal{L}(\uu)$, then $\phi(w)1 \in \mathcal{L}(\uu)$.
\item If $u$ is a fractional root of $w$,  then $\phi(u)$ is a fractional root of $\phi(w)1$.
\end{enumerate}
We will construct two sequences $\bigl(u^{(n)}\bigr)$ and $\bigl(w^{(n)}\bigr)$  of words belonging to $\mathcal{L}(\uu)$.
We start with $u^{(0)}=1$ and $w^{(0)}=111\in \mathcal{L}(\uu)$ and for each $n\in \mathbb{N}$ we define
$$ w^{(n+1)} = \phi(w^{(n)})1 \qquad \text{and} \qquad  u^{(n+1)} = \phi(u^{(n)})\,.$$
Note that $u^{(0)}=1$ is a fractional root of $w^{(0)}=111$.
Because of the property (2), the word $u^{(n)}$ is a fractional root of $w^{(n)}$ for each $n\in \N$.
Moreover, the specific form of the morphism $\phi$ implies $| u^{(n+1)}| = 2| u^{(n)}|$ and $| w^{(n+1)}| =2 | w^{(n)}| +1$.
It gives $| u^{(n)}| = 2^n$ and $ | w^{(n)}|  = 2^{n+2} -1$.
Therefore, $\ind_\uu(u^{(n)}) \geq \frac{2^{n+2} -1}{2^n} \rightarrow  4$.
We may conclude that $\CR(\uu) = 4$.
\end{exam}

\section{Return words to bispecial factors of Sturmian sequences} \label{S_retwordsBS_Sturmian}

The main goal of this article is to describe the critical exponent and the recurrence function of CS Rote sequences. Proposition \ref{Prop_SturmRote} and Theorem \ref{Lem_PrevodKritExpoRoteNaSturm} transform the first task to the computation of the indices of return words to bispecial factors in the associated Sturmian sequences.

This is a preparatory section for this computation.
We introduce the directive sequence of a standard Sturmian sequence and recall some known results on bispecial factors, their return words, and derived sequences.
It allows us to describe the longest factor of $\uu$ with the fractional root $u$, where $u$ is any return word to a bispecial factor of a Sturmian sequence $\uu$ (Lemma \ref{mocniny}).
Furthermore, we explain how to express the lengths of these factors explicitly (Proposition \ref{ParikhRSB}), and eventually in Section \ref{S_CriticalExponentRote}, we determine the indices.

First, we recall that a binary sequence $\uu \in \{0,1\}^\N$ is Sturmian if it has the factor complexity $\mathcal{C}_\uu(n) = n+1$ for all $n \in \N$.
If both sequences $0\uu$ and $1\uu$ are Sturmian, then $\uu$ is called a \textit{standard Sturmian sequence}.
It is well-known that for any Sturmian sequence there exists a unique standard Sturmian sequence with the same language.
Since all properties which we are interested in (indices of factors, critical exponent, special factors, return words, recurrence function) depend only on the language of the sequence, we restrict ourselves to standard Sturmian sequences without loss of generality.

In the sequel, we use the characterization of standard Sturmian sequences by their directive sequences.
To introduce them, we define two morphisms
$$
G = \
\left\{ \, \begin{aligned}
0 &\to 10 \\
1 &\to 1\,
\end{aligned}\right.
 \ \ \ \ \ \ \ \text{and}  \ \ \ \ \ \ \ \ \ \
 D = \
\left\{ \, \begin{aligned}
0 &\to 0 \\
1 &\to 01\,
\end{aligned} \right.
$$
with the corresponding matrices
$$
M_G = \left(\begin{array}{cc}
1 & 0\\
1 & 1\\
\end{array} \right)\quad \quad \ \text{and} \quad \quad \quad
M_D = \left(\begin{array}{cc}
1 & 1\\
0 & 1\\
\end{array} \right)\,.
$$

Let us note that $G = E \circ F$ and $D = F \circ E$, where $E$ is the morphism which exchanges letters, i.e., $E: 0 \to 1$, $1 \to 0$, and $F$ is the Fibonacci morphism, i.e., $F: 0 \to 01$, $1 \to 0$.

\begin{prop}[\cite{JuPi}] \label{Lem_Standard}
For every standard Sturmian sequence $\uu$ there is a uniquely given
 sequence ${\bf \Delta} = \Delta_0\Delta_1\Delta_2 \cdots \in \{G, D\}^\N$ of morphisms and a sequence $(\uu^{(n)})_{n \geq 0}$ of standard Sturmian sequences such that
$$
\uu = {\Delta_0\Delta_1 \ldots \Delta_{n-1}}(\uu^{(n)})\,   \ \text{for every } \ n \in \N\,.
$$
Moreover, the sequence ${\bf \Delta}$ contains infinitely many letters $G$ and infinitely many letters $D$, i.e., $$
{\bf \Delta} =  G^{a_1}D^{a_2}G^{a_3}D^{a_4} \cdots \ \text{ or } \ {\bf \Delta} = D^{a_1}G^{a_2}D^{a_3}G^{a_4} \cdots \quad \text{for some sequence } (a_i)_{i \geq 1} \text{ of positive integers} \,.
$$
The sequence ${\bf \Delta}$ is called the \emph{directive sequence} of $\uu$.
\end{prop}

\begin{rem} \label{Rem_Zamena}
Let us note that  $\uu$ has the directive sequence $G^{a_1}D^{a_2}G^{a_3}D^{a_4} \cdots$ if and only if $E(\uu)$ has the directive sequence $D^{a_1}G^{a_2}D^{a_3}G^{a_4} \cdots$. Obviously, both sequences $\uu$ and $E(\uu)$ have the same structure up to the exchange of letters $0 \leftrightarrow 1$.
In particular, any Sturmian sequence with the directive sequence $G^{a_1}D^{a_2}G^{a_3}D^{a_4} \cdots$ can be written as a concatenation of the blocks $1^{a_1}0$ and $1^{a_1 + 1}0$, while any Sturmian sequence with the directive sequence $D^{a_1}G^{a_2}D^{a_3}G^{a_4} \cdots$ can be written as a concatenation of the blocks $0^{a_1}1$ and $0^{a_1 + 1}1$.

By Vuillon's result \cite{Vui01}, every factor of any Sturmian sequence has exactly two return words.
Thus for a given bispecial factor $b$ of  $\uu$, we usually denote the more and the less frequent return word to $b$ in  $\uu$ by $r$ and $s$.
In this notation, the sequence $\uu$ can be decomposed into the blocks $r^ks$ and $r^{k+1}s$ for some $k \in \N, k \geq 1$.
\end{rem}

We need to know how bispecial factors and their return words change under the application of morphisms $G$ and $D$.
The following description can be found in \cite{MePeVu}, where several partial statements from \cite{KlMePeSt18} are accumulated.

\begin{lem} \label{Lem_ImageG}
Let $\uu', \uu$ be standard Sturmian sequences such that $\uu = G(\uu')$.
\begin{itemize}
\item[\textit{(i)}] For every bispecial factor $b'$ of $\uu'$, the factor $b = G(b')1$ is a bispecial factor of $\uu$.
\item[\textit{(ii)}] Every bispecial factor $b$ of $\uu$ which is not empty can be written as $b = G(b')1$ for a uniquely given bispecial factor $b' \in \mathcal{L}(\uu')$.
\item[\textit{(iii)}] The words $r', s'$ are return words to a bispecial prefix $b'$ of $\uu'$ if and only if $r = G(r'), s =G(s')$ are return words to a bispecial prefix $b = G(b')1$ of $\uu$. Moreover, the derived sequences satisfy $\dd_\uu(b) = \dd_{\uu'}(b')$.
\end{itemize}
\end{lem}

\begin{lem} \label{Lem_ImageD}
Let $\uu', \uu$ be standard Sturmian sequences such that $\uu = D(\uu')$.
\begin{itemize}
\item[\textit{(i)}] For every bispecial factor $b'$ of $\uu'$, the factor $b = D(b')0$ is a bispecial factor of $\uu$.
\item[\textit{(ii)}] Every bispecial factor $b$ of $\uu$ which is not empty can be written as $b = D(b')0$ for a uniquely given bispecial factor $b' \in \mathcal{L}(\uu')$.
\item[\textit{(iii)}] The words $r', s'$ are return words to a bispecial prefix $b'$ of $\uu'$ if and only if $r =D(r'), s = D(s')$ are return words to a bispecial prefix $b = D(b')0$ of $\uu$. Moreover, the derived sequences satisfy $\dd_\uu(b) = \dd_{\uu'}(b')$.
\end{itemize}
\end{lem}

Any prefix of a standard Sturmian sequence is a left special factor.
Moreover, a factor of a standard Sturmian sequence $\uu$ is bispecial if and only if it is a palindromic prefix of $\uu$.
Therefore, we can order the bispecial factors of a given standard Sturmian sequence $\uu$ by their lengths: we start with the empty word $\varepsilon$, which is the $0^{th}$ bispecial factor, then the first letter of $\uu$ is the $1^{st}$ bispecial factor of $\uu$ etc.

\begin{rem}\label{Rem_DerivedSeq}
If $\uu$ has the directive sequence $\Delta_0\Delta_1\Delta_2 \cdots \in \{G, D\}^\N$, the derived sequence $\dd_{\uu}(b)$ to the $n^{th}$ bispecial factor $b$ of $\uu$ has the directive sequence $\Delta_{n}\Delta_{n+1}\Delta_{n+2} \cdots$.
Indeed, we denote the sequence with the directive sequence $\Delta_{n}\Delta_{n+1} \Delta_{n+2}\cdots$ by $\uu'$.
It has the bispecial factor $\varepsilon$ and by the definition $\dd_{\uu'}(\varepsilon) = \uu'$.
If we apply $n$ times Lemmas \ref{Lem_ImageG} or \ref{Lem_ImageD}, we get $\dd_{\uu}(b) = \dd_{\uu'}(\varepsilon) = \uu'$.
\end{rem}

Let us formulate a direct consequence of the relation \eqref{multi} and  Lemmas \ref{Lem_ImageG} and \ref{Lem_ImageD}.

\begin{coro}\label{obrazyBS}
Let $k,h\in \N$.  Let $b'$ be the $k^{th}$ bispecial factor of a standard Sturmian sequence $\uu'$  and $u'$  be a return word to $b'$ in $\uu'$.  Let ${\bf \Delta} = \Delta_0\Delta_1 \Delta_2 \cdots \in\{G, D\}^\N$ be the directive sequence of $
\uu'$.
\begin{enumerate}
 \item If $\uu  = G^h(\uu')$, then the $(k+h)^{th}$ bispecial factor $b$ of $\uu$  and a  return word $u$  to $b$ satisfy
$$
\vec{V}(b) = \left(\!\!\begin{array}{cc}1&0\\ h&1 \end{array}\!\!\right)\vec{V}(b') + h \left(\!\!\begin{array}{c}0\\ 1 \end{array}\!\!\right)  \quad \text{and}\quad  \vec{V}(u) = \left(\!\!\begin{array}{cc}1&0\\ h&1 \end{array}\!\!\right)\vec{V}(u')\,.
$$
The directive sequence of $\uu$ is $G^h\Delta_0\Delta_1 \Delta_2 \cdots$.
\medskip

\item If $\uu  =D^h(\uu')$, then the $(k+h)^{th}$ bispecial factor $b$ of $\uu$ and a  return word $u$  to $b$ satisfy
$$
\vec{V}(b) = \left(\!\!\begin{array}{cc}1&h\\ 0&1 \end{array}\!\!\right)\vec{V}(b') + h \left(\!\!\begin{array}{c}1\\ 0 \end{array}\!\!\right) \quad \text{and}\quad  \vec{V}(u) = \left(\!\!\begin{array}{cc}1&h\\ 0&1 \end{array}\!\!\right)\vec{V}(u') \,.
$$
The directive sequence of $\uu$ is $D^h\Delta_0\Delta_1 \Delta_2 \cdots$.
\end{enumerate}
\end{coro}

As we have seen in Remark \ref{rem_CriticalExponentByReturnWords}, when determining the critical exponent it suffices to take into account only bispecial factors whose fractional roots are equal to its return words.
Lemma \ref{mocniny} says that all bispecial factors of a Sturmian sequence are of this type, and moreover, it enables one to determine the indices of their return words.
The first auxiliary statement is a slightly strengthened variant of Observation \ref{pridano} for the morphisms $G$ and $D$.

\begin{obs} \label{Lem_FractionalRoot}
Let $\uu$ be a binary sequence and let $u \in \mathcal{L}(\uu)$.
If $z$ is the longest factor in $\mathcal{L}(\uu)$ with the fractional root $u$, then $G(z)1$ is the longest factor in $\mathcal{L}(G(\uu))$ with the fractional root $G(u)$ and, analogously, $D(z)0$ is the longest factor in $\mathcal{L}(D(\uu))$ with the fractional root $D(u)$.
\end{obs}

\begin{lem}\label{mocniny}
Let $b$ be a bispecial factor of a standard Sturmian sequence $\uu$.
Let $r$ and $s$ be the return words to $b$ in $\uu$ and let $k \in \N$, $k \geq 1$, be such that $\uu$ is concatenated from the blocks $r^ks$ and $r^{k+1}s$.
Then $r^{k+1}b$ is the longest factor of $\uu$ with the fractional root $r$ and $sb$ is the longest factor of $\uu$ with the fractional root~$s$.
\end{lem}

\begin{proof}
We proceed by induction on the length of $b$.
Without loss of generality, we assume that $\uu$ has the directive sequence ${\bf \Delta} = G^{a_1}D^{a_2}G^{a_3}D^{a_4} \cdots$.

The bispecial factor $b = \varepsilon$ has the return words $r= 1, s = 0$, and by Remark \ref{Rem_Zamena}, $\uu$ is concatenated from the blocks $1^{a_1}0 = r^{a_1}s$ and $1^{a_1 +1}0 = r^{a_1 + 1}s$.
Clearly, $r^{a_1+1}b = 1^{a_1+1}$ is the longest factor of $\uu$ with the fractional root $1$.
Similarly, $sb = 0$ is the longest factor of $\uu$ with the fractional root $0$.

Let $b$ be a bispecial factor of $\uu$ with $|b| \geq 1$ and let $\uu$ be concatenated from the blocks $r^{k}s$ and $r^{k+1}s$ for the return words $r,s$ to $b$ in $\uu$ and some $k \in \N$, $k \geq 1$.
By Proposition \ref{Lem_Standard}, there is a unique standard Sturmian sequence $\uu'$ such that $\uu = G(\uu')$.
By Lemmas \ref{Lem_ImageG} and \ref{Lem_ImageD}, there is a unique bispecial factor $b'$ of $\uu'$ with the return words $r'$ and $s'$ such that $b = G(b')1$, $r = G(r')$, and $s = G(s')$.
Moreover, $\uu'$ is concatenated from the blocks $(r')^{k}s'$ and $(r')^{k+1}s'$.
Clearly, $|b'| < |b|$ and so by the induction hypothesis, the words $(r')^{k+1}b'$ and  $s'b'$ are the longest factors of $\uu'$ with the fractional root $r'$ and $s'$, respectively.
But then by Observation \ref{Lem_FractionalRoot}, the words $r^{k+1}b = G((r')^{k+1}b')1$ and $sb = G(s'b')1$ are the longest factors of $\uu$ with the fractional root $r = G(r')$ and $s = G(s')$, respectively.
\end{proof}

Having in mind our goal to describe the critical exponent of any CS Rote sequence and
Theorem \ref{Lem_PrevodKritExpoRoteNaSturm}, we need to determine the indices of return words to bispecial factors in standard Sturmian sequences, i.e., the lengths of factors from Lemma \ref{mocniny}. We also want to distinguish, which of these return words are (un)stable.
Both of these tasks can be solved using the Parikh vectors of the relevant bispecial factors and their return words.
We deduce the explicit formulae for the needed Parikh vectors in Proposition \ref{ParikhRSB}.
For this purpose, we adopt the following notation.

\begin{nota}\label{Theta}
To a standard Sturmian sequence $\uu$  with the directive sequence ${\bf \Delta} = G^{a_1}D^{a_2}G^{a_3}D^{a_4} \cdots $ or ${\bf \Delta} = D^{a_1}G^{a_2}D^{a_3}G^{a_4} \cdots $ we assign an irrational number $\theta \in (0,1)$ with the continued fraction expansion
$$\theta = [0, a_1,a_2, a_3, \ldots ]\,.$$
For every $N\in \mathbb{N}$, we denote the $N^{th}$ \textit{convergent} to the number  $\theta$  by $\tfrac{p_N}{q_N}$ and  the $N^{th}$ convergent to the number $\frac{\theta}{1+\theta}$ by $\frac{p'_N}{q'_N}$.
\end{nota}

\begin{rem}\label{SlizeneZlomky}
Let us recall some basic properties of convergents. They can be found in any number theory textbook, e.g., \cite{Hensley}.

\begin{enumerate}
\item  The sequences $({p_N})$, $({q_N})$, and $({q'_N})$ fulfil the same  recurrence relation for all $N \in \N, N \geq 1$, namely
$$  X_{N} = a_{N} X_{N-1} +X_{N-2}\,, $$
but they differ in their initial values: $p_{-1} = 1, p_0 = 0$;  $q_{-1} = 0, q_0 = 1$;  $q'_{-1} =q'_0 = 1$. It implies for all $N \in \N$
$$
p_N +q_N = q'_N\,.
$$

\medskip

\item For all $N \in \N, N \geq 1$, we have

 \medskip
$
\left(\!\!\begin{array}{cc}1&0\\ a_1&1 \end{array}\!\!\right) \left(\!\!\begin{array}{cc}1&a_2\\ 0&1 \end{array}\!\!\right) \cdots
\left(\!\!\begin{array}{cc}1&0\\ a_{2N-1}&1 \end{array}\!\!\right) \left(\!\!\begin{array}{cc}1&a_{2N}\\ 0&1 \end{array}\!\!\right) =
\left(\!\!\begin{array}{cc}p_{2N-1}&p_{2N}\\ q_{2N-1}&q_{2N} \end{array}\!\!\right)\,;
$

\medskip

$
\left(\!\!\begin{array}{cc}1&0\\ a_1&1 \end{array}\!\!\right) \left(\!\!\begin{array}{cc}1&a_2\\ 0&1 \end{array}\!\!\right) \cdots  \left(\!\!\begin{array}{cc}1&a_{2N-2}\\ 0&1 \end{array}\!\!\right)
\left(\!\!\begin{array}{cc}1&0\\ a_{2N-1}&1 \end{array}\!\!\right)=
\left(\!\!\begin{array}{cc}p_{2N-1}&p_{2N-2}\\ q_{2N-1}&q_{2N-2} \end{array}\!\!\right)\,.
$
\end{enumerate}

\medskip

\end{rem}

\begin{rem}\label{density}
For the description of a standard Sturmian sequence $\uu$ we use the number $\theta$.
Usually, a standard Sturmian sequence is characterized by the so-called \textit{slope}, which is equal to the density of the letter 1 in the sequence $\uu$. In our notation, the  slope of $\uu$  is $\frac{\theta}{1+\theta} = [0, 1+a_1, a_2, a_3, \ldots ]$ if the directive sequence ${\bf \Delta}$ starts with $D$, otherwise the slope is $\frac{1}{1+\theta}  = [0, 1, a_1, a_2, a_3, \ldots ]$.
\end{rem}

In the sequel, we will need two auxiliary statements on convergents $\tfrac{p_N}{q_N}$ to $\theta$.

\begin{lem}\label{Lem_Convergents}
For all $N \in \N$ we have
 $$\left(\!\!\begin{array}{c}p_N\\ q_N\end{array}\!\!\right)\not =\left(\!\!\begin{array}{c}0\\ 0\end{array}\!\!\right) \mod 2 \qquad \text{and} \qquad \left(\!\!\begin{array}{c}p_N\\ q_N\end{array}\!\!\right)\not =\left(\!\!\begin{array}{c}p_{N-1}\\ q_{N-1}\end{array}\!\!\right) \mod 2 \,.$$
\end{lem}

\begin{proof}
The first statement is a consequence of the fact that $p_N$ and $q_N$ are coprime.  We show the second statement by contradiction.
 Assume that there exists  $K \in \N$  such that $\left(\!\!\begin{array}{c}p_K\\ q_K\end{array}\!\!\right) =\left(\!\!\begin{array}{c}p_{K-1}\\ q_{K-1}\end{array}\!\!\right) \mod 2$. Let $K$ denote  the smallest integer with this property. As  $q_{-1}= 0$ and $q_{0} =1$, necessarily,  $K>0$.  Using the recurrence relation satisfied by the convergents, we  can  write
$$
\left(\!\!\begin{array}{c}p_K\\ q_K\end{array}\!\!\right) = a_{K}\left(\!\!\begin{array}{c}p_{K-1}\\ q_{K-1}\end{array}\!\!\right)  + \left(\!\!\begin{array}{c}p_{K-2}\\ q_{K-2}\end{array}\!\!\right) =\left(\!\!\begin{array}{c}p_{K-1}\\ q_{K-1}\end{array}\!\!\right)   \mod 2\,.
$$

If $a_K$ is even, then the previous equation gives $ \left(\!\!\begin{array}{c}p_{K-2}\\ q_{K-2}\end{array}\!\!\right) =\left(\!\!\begin{array}{c}p_{K-1}\\ q_{K-1}\end{array}\!\!\right)   \mod 2$, which is  a contradiction with the minimality of $K$.

If $a_K$ is odd, then the previous equation gives $ \left(\!\!\begin{array}{c}p_{K-2}\\ q_{K-2}\end{array}\!\!\right) =\left(\!\!\begin{array}{c}0\\ 0\end{array}\!\!\right)   \mod 2$, which is a contradiction with the first statement.
\end{proof}

\begin{lem}\label{matice_konvergenty}
For all $N \in \N$, $N \geq 1$, we have
$$
a_N    \left(\!\!\begin{array}{c}p_{N-1}\\ q_{N-1} \end{array}\!\!\right)  + a_{N-1}  \left(\!\!\begin{array}{c}p_{N-2}\\ q_{N-2} \end{array}\!\!\right) + \cdots +   a_2\left(\!\!\begin{array}{c}p_{1}\\ q_{1} \end{array}\!\!\right)  +   a_1\left(\!\!\begin{array}{c}p_{0}\\ q_{0} \end{array}\!\!\right) =
  \left(\!\!\begin{array}{c}p_{N}\\ q_{N} \end{array}\!\!\right)   +   \left(\!\!\begin{array}{c}p_{N-1}\\ q_{N-1} \end{array}\!\!\right)   -   \left(\!\!\begin{array}{c}1\\1 \end{array}\!\!\right)\,.
$$
\end{lem}

\begin{proof}
It can be easily proved by induction on $N$.
\end{proof}

The Parikh vectors of the bispecial factors of $\uu$  and the corresponding return words  can be easily expressed using the convergents $\frac{p_N}{q_N}$ to $\theta$.
We will use these expressions essentially in the next sections.

\begin{prop}\label{ParikhRSB}
Let $b$ be the $n^{th}$ bispecial factor of $\uu$ with the directive sequence $G^{a_1}D^{a_2}G^{a_3}D^{a_4}\cdots $.
We denote the more and the less frequent return word to $b$ in $\uu$ by  $r$ and $s$, respectively.
Put $a_0=0$ and  write $n$ in the form $n =m+ a_0+ a_1+a_2+\cdots +a_N $ for a unique $N \in \N$ and $0\leq m<a_{N+1}$.
Then

\medskip
\begin{enumerate}
\item $\vec{V}(r)  = \left(\!\!\begin{array}{c}p_N\\ q_N\end{array}\!\!\right)$;

\medskip
\item $ \vec{V}(s)  = \left(\!\!\begin{array}{c}m\,p_{N}+p_{N-1}\\ m\,q_{N}+q_{N-1}\end{array}\!\!\right)$;

\medskip
\item $\vec{V}(b)= (m+1)\left(\!\!\begin{array}{c}p_{N}\\ q_{N} \end{array}\!\!\right)   +   \left(\!\!\begin{array}{c}p_{N-1}\\ q_{N-1} \end{array}\!\!\right)   -   \left(\!\!\begin{array}{c}1\\1 \end{array}\!\!\right)  $.
\end{enumerate}
\end{prop}

\begin{proof}
First we suppose that $N$ is even and we denote the standard Sturmian sequence with the directive sequence $G^{a_{N+1}}D^{a_{N+2}}G^{a_{N+3}}\cdots$ by $\uu'$.
By Remark \ref{Rem_Zamena}, $\uu'$ is concatenated from the blocks $1^{a_{N+1}}0$ and $1^{a_{N+1}+1}0$.
Thus its $m^{th}$ bispecial factor is $b'=1^m$ and the return words to $b'$ in $\uu'$ are $r' =1$ and  $s' = 1^m0$.

By Lemmas \ref{Lem_ImageG} and \ref{Lem_ImageD} and Remark \ref{Rem_DerivedSeq},
$$
r = G^{a_1}D^{a_2}\cdots G^{a_{N-1}}D^{a_{N}}( r')\qquad \text{and}\qquad s= G^{a_1}D^{a_2}\cdots G^{a_{N-1}}D^{a_{N}}( s')\,.
$$
By Corollary \ref{obrazyBS} and  Lemma \ref{matice_konvergenty}, the Parikh vectors of $r$ and $s$ satisfy
\begin{align*}
\vec{V}(r) &= \left(\!\!\begin{array}{cc}1&0\\ a_1&1 \end{array}\!\!\right) \left(\!\!\begin{array}{cc}1&a_2\\ 0&1 \end{array}\!\!\right) \cdots
\left(\!\!\begin{array}{cc}1&0\\ a_{N-1}&1 \end{array}\!\!\right) \left(\!\!\begin{array}{cc}1&a_{N}\\ 0&1 \end{array}\!\!\right) \left(\!\!\begin{array}{c}0\\ 1\end{array}\!\!\right) =
\left(\!\!\begin{array}{c}p_{N}\\ q_{N} \end{array}\!\!\right) \,; \\
\vec{V}(s) &= \left(\!\!\begin{array}{cc}1&0\\ a_1&1 \end{array}\!\!\right) \left(\!\!\begin{array}{cc}1&a_2\\ 0&1 \end{array}\!\!\right) \cdots
\left(\!\!\begin{array}{cc}1&0\\ a_{N-1}&1 \end{array}\!\!\right) \left(\!\!\begin{array}{cc}1&a_{N}\\ 0&1 \end{array}\!\!\right) \left(\!\!\begin{array}{c}1\\ m\end{array}\!\!\right) =
\left(\!\!\begin{array}{c}m\,p_{N}+p_{N-1}\\ m\,q_{N}+q_{N-1} \end{array}\!\!\right) \,.
\end{align*}
To find the Parikh vector of $b$ we start with the bispecial factor $b' = 1^m$  and $N$ times apply Corollary \ref{obrazyBS}.
Eventually, we rewrite the arising products of matrices by Lemma \ref{matice_konvergenty} and we get
\begin{align*}
\vec{V}(b)  &= m\, \left(\!\!\begin{array}{c}p_{N}\\ q_{N} \end{array}\!\!\right)  + a_N    \left(\!\!\begin{array}{c}p_{N-1}\\ q_{N-1} \end{array}\!\!\right)  + a_{N-1}  \left(\!\!\begin{array}{c}p_{N-2}\\ q_{N-2} \end{array}\!\!\right) + \cdots +   a_2\left(\!\!\begin{array}{c}p_{1}\\ q_{1} \end{array}\!\!\right)  +   a_1\left(\!\!\begin{array}{c}p_{0}\\ q_{0} \end{array}\!\!\right) \,.
\end{align*}
This together with Lemma \ref{matice_konvergenty} implies the statement of Item 3 for $N$ even.
The proof for $N$ odd is analogous.
\end{proof}

\begin{rem}\label{Rem_GandDexchanged}
If we assume in Proposition~\ref{ParikhRSB} that $\uu$ has the directive sequence ${\bf \Delta} = D^{a_1}G^{a_2}D^{a_3}G^{a_4}\cdots $, then by Remark~\ref{Rem_Zamena}, the coordinates of the Parikh vectors will be exchanged, i.e.,
$$
\vec{V}(r)  = \left(\!\!\begin{array}{c}q_N\\ p_N\end{array}\!\!\right), \ \vec{V}(s)  = \left(\!\!\begin{array}{c} m\,q_{N}+q_{N-1} \\ m\,p_{N}+p_{N-1}\end{array}\!\!\right), \ \text{and} \ \vec{V}(b)= (m+1)\left(\!\!\begin{array}{c}q_{N}\\ p_{N} \end{array}\!\!\right)   +   \left(\!\!\begin{array}{c}q_{N-1}\\ p_{N-1} \end{array}\!\!\right)   -   \left(\!\!\begin{array}{c}1\\1 \end{array}\!\!\right).
$$
\end{rem}

\section{The critical exponent of CS Rote sequences} \label{S_CriticalExponentRote}

We are going to give an explicit formula for the critical exponent of a CS Rote sequence $\vv$.
We will use Theorem \ref{Lem_PrevodKritExpoRoteNaSturm} which requires the knowledge of the indices of return words to bispecial factors in the Sturmian sequence $\S(\vv)$.
It is well-known that there is a unique standard Sturmian sequence $\uu$ such that both $\S(\vv)$ and $\uu$ have the same language.
Since the critical exponent depends only on the language, we can work with the standard Sturmian sequence $\uu$ instead of $\S(\vv)$.

In the following proposition and theorem, we use  Notation \ref{Theta}.

\begin{prop}\label{Prop_KritExpo}
Let $\uu$  be a standard Sturmian sequence with the directive sequence $G^{a_1}D^{a_2}G^{a_3}D^{a_4}\cdots$ or $D^{a_1} G^{a_2} D^{a_3}G^{a_4} \cdots $ and let $n \in \N$ be given.
We put $a_0 = 0$ and we denote the more and the less frequent return word to the $n^{th}$ bispecial factor of $\uu$ by $r$ and $s$, respectively.
\begin{enumerate}
\item If $n =m+ a_0+a_1+a_2+\cdots +a_N $, where $0\leq m<a_{N+1}$, \ \ then  \ \
$\emph{ind}(r) = a_{N+1} +2 + \frac{q'_{N-1}-2}{q'_N}$.
\item If $n =a_0+ a_1+a_2+\cdots +a_N$, \ \ then  \ \
$ \emph{ind}(s) = a_{N} +2 + \frac{q'_{N-2}-2}{q'_{N-1}}$.
\item  If $n =m+a_0+ a_1+a_2+\cdots +a_N $, where $0<m<a_{N+1}$, \ \ then  \ \
$\emph{ind}(s) = 2 + \frac{q'_{N}-2}{ m\,q'_{N} +q'_{N-1}}$.
\end{enumerate}
\end{prop}

Let us comment on what is meant by $q'_{-2}$ in the case $N=0$ in Item 2: we define $q'_{-2}$ to satisfy  the recurrence relation $1=q'_{0} = a_0q'_{-1} + q'_{-2} = q'_{-2 }$.

\begin{proof}
We assume that $N$ is even (the case of $N$ odd is analogous). Moreover, we assume that $\uu$ has the directive sequence $G^{a_1} D^{a_2} G^{a_3} D^{a_4} \cdots $.

Since $b$ is the $n^{th}$ bispecial factor  in $\uu$, the derived sequence $\dd_{\uu}(b)$ is standard Sturmian with  the directive sequence $G^{k}D^{a_{N+2}}G^{a_{N+3}}\cdots $, where  $k=a_{N+1}-m$, see Corollary \ref{obrazyBS}.
By Remark \ref{Rem_Zamena}, $\dd_{\uu}(b)$  is a concatenation of the blocks $1^k0$ and $1^{k+1}0$. Therefore, the sequence $\uu$ is concatenated from the blocks $r^ks$ and $r^{k+1}s$, where $r$ and $s$ are the return words to $b$.
By Lemma \ref{mocniny}, the factor $r^{k+1}b$ is the longest factor of $\uu$ with the fractional root $r$ and $sb$ is the longest factor of $\uu$ with the fractional root $s$.
In other words, $r^{\ind(r)}=r^{k+1}b$ and $s^{\ind(s)}= sb$.
By Proposition  \ref{ParikhRSB} and Remark \ref{SlizeneZlomky}, we have
$$|r| = p_N +q_N = q'_N,  \quad  \ |s| = mq_N' + q'_{N-1} \,, \quad \text{and} \ \quad  |b| =  (m+1)q'_N  + q'_{N-1} - 2\,. $$
As $|r^{k+1}b| = (k+1)|r| + |b| = (a_{N+1}-m+1) |r|+|b|  = (a_{N+1}+2) q'_N + q'_{N-1} -2$,  we get $$ \ind(r)=\tfrac{|r^{k+1}b|}{|r|} = a_{N+1} +2 + \tfrac{q'_{N-1}-2}{q'_N}\,.$$
As $|sb| =  (2m+1)q_N' + 2q'_{N-1} - 2$, we  get for $m=0$
$$\ind(s)=\tfrac{|sb|}{|s|} = \tfrac{q_N' + 2q'_{N-1} - 2}{ q'_{N-1}  } = \tfrac{a_Nq'_{N-1}+q'_{N-2}+ 2q'_{N-1} - 2}{ q'_{N-1}  }  = a_{N} +2 + \tfrac{q'_{N-2}-2}{q'_{N-1}}\,$$
and  for $m>0$
$$ \ind(s)= \tfrac{|sb|}{|s|} = \tfrac{(2m+1)q_N' + 2q'_{N-1} - 2}{m\,q'_{N} +q'_{N-1}} =  2 + \tfrac{q'_{N}-2}{ m\,q'_{N} +q'_{N-1}}\,.$$

If the directive sequence equals $D^{a_1} G^{a_2} D^{a_3} G^{a_4} \cdots$, only the coordinates of the Parikh vectors of $r$, $s$, $b$ are exchanged (see Remark \ref{Rem_GandDexchanged}).
\end{proof}

\begin{thm} \label{T_KritExpo}
Let $\vv$ be a CS Rote sequence and let $\uu$ be the standard Sturmian sequence such that $\mathcal{L}(\S(\vv)) = \mathcal{L}(\uu)$.
Then
$
\emph{cr}(\vv) = \sup (M_1 \cup M_2 \cup M_3)\,, \text{ where }
$
\begin{align*}
 \ \ M_1 &= \left\{a_{N+1} + 2 + \frac{q'_{N-1}-1}{q'_N} : \ q_N \text{ is even, } N\in \N\right\}\,;  \\
M_2 &= \left\{\frac{a_{N+1} + 2}{2} + \frac{q'_{N-1}-1}{2 q'_N} : \ q_N \text{ is odd,  } N \in \N\right\}\,; \\
M_3 &= \left\{2 + \frac{q'_{N}-1}{q'_{N-1} + q'_N} : \ q_{N-1} ,  q_{N} \text{ are  odd and }  a_{N+1}>1, N \geq 1\right\} \,
\end{align*}
if the directive sequence of   $\uu$  is $G^{a_1}D^{a_2}G^{a_3}D^{a_4}\cdots$, and
\begin{align*}
M_1 &= \left\{a_{N+1} + 2 + \frac{q'_{N-1}-1}{q'_N} : \ p_N \text{ is even, } N\in \N\right\}\,; \\
M_2 &= \left\{\frac{a_{N+1} + 2}{2} + \frac{q'_{N-1}-1}{2 q'_N} : \ p_N \text{ is odd,  } N \in \N\right\}\,; \\
M_3 &= \left\{2 + \frac{q'_{N}-1}{q'_{N-1} + q'_N} : \ p_{N-1} ,  p_{N} \text{ are  odd and }  a_{N+1}>1, N \geq 1\right\} \,
\end{align*}
if  the directive sequence of   $\uu$  is  $D^{a_1}G^{a_2}D^{a_3}G^{a_4}\cdots$.
\end{thm}

\begin{proof}
Let us recall that every CS Rote sequence is uniformly recurrent and aperiodic.
In addition, to a CS Rote sequence $\vv$ we can always find a unique standard Sturmian sequence $\uu$ such that $\uu$ has the same language as $\S(\vv)$.
It is important to realize that Theorem \ref{Lem_PrevodKritExpoRoteNaSturm} holds for the pair $\vv$ and $\uu$, too.

First, we assume that $\uu$ has the directive sequence $G^{a_1}D^{a_2}G^{a_3}D^{a_4}\cdots$.
We compute the suprema of the sets $A_1$ and $A_2$ defined in Theorem \ref{Lem_PrevodKritExpoRoteNaSturm}, since by this theorem, $\CR(\vv) = \sup(A_1\cup A_2)$.

Let us decompose $A_  1$ into $A_1 = \bigcup_{N =0}^{\infty} A_1^{(N)}$, where
$$
A_1^{(N)} = \Bigl\{\ind_\uu(u)  + \tfrac{1}{|u|} :\  u  \text{ is a stable return word  to the $n^{th}$ bispecial f. of } \uu, \   \sum_{k=0}^N a_k \leq n  < \sum_{k=0}^{N+1} a_k  \Bigr\}\,.
$$
By definition, a word $u$ is stable if the number of ones occurring in $u$ is even, i.e., the second component of its Parikh vector $\vec{V}(u)$ is even.  Combining   Propositions \ref{ParikhRSB}  and \ref{Prop_KritExpo}, we obtain that $A_1^{(N)} $ contains
\begin{itemize}
\item $a_{N+1} +2 + \frac{q'_{N-1}-1}{q'_N} $  if   $q_N$ is even,
\item $ a_{N} +2 + \frac{q'_{N-2}-1}{q'_{N-1}}$  if  $q_{N-1}$ is even,
\item the subset
$ B_1^{(N)} =\{2 + \frac{q'_{N}-1}{ m\,q'_{N} +q'_{N-1}} : m\,q_{N} +q_{N-1} \text{ even}, \  0< m < a_{N+1}\}$.
\end{itemize}
First we look at $A_1^{(0)} $. Since $q_0 = 1$ is odd, $q_{-1}=0$ is even,  $a_0 =0$, $q'_{0} = q'_{-1}= q'_{-2}= 1$, we get $A_1^{(0)} = \{2\}$.
Since we know that $\CR(\vv) > 2$, we can consider only $N\geq 1$.
Let us note that all elements in $B_1^{(N)}$ are strictly less than $3$.
If $q_N$ or $q_{N-1}$ is even, the set  $A_1^{(N)} $ contains an element $\geq 3$ and the set $B_1^{(N)}$ does not play any role for $\sup
A_1$.
If both $q_N$ and $q_{N-1}$ are odd, there is an element in $B_1^{(N)}$  only for odd $m  < a_{N+1}$, and obviously, $\sup  B_1^{(N)}$ is attained for $m=1$ (if $a_{N+1} = 1$ the set is empty).
Together it gives $\sup A_1 =\sup( M_1\cup M_3)$.

Analogously we define the sets $A_2^{(N)}$ for unstable return words.
Then $A_2^{(N)}$ consists of
\begin{itemize}
\item $\tfrac12\big(a_{N+1} +2 + \frac{q'_{N-1}-1}{q'_N}\bigr)$ if $q_N$ is odd,
\item  $\tfrac12\big( a_{N} +2 + \frac{q'_{N-2}-1}{q'_{N-1}}\bigr) $  if   $q_{N-1}$ is odd,
\item  the subset
$ B_2^{(N)} =\{ \tfrac12\big(2 + \frac{q'_{N}-1}{ m\,q'_{N} +q'_{N-1}} \bigr) : m\,q_{N} +q_{N-1} \text{ odd}, \  0< m < a_{N+1}\}$.
\end{itemize}
We easily compute that $\sup A_2^{(0)} = \tfrac12(a_1+2)$.
All elements in $B_2^{(N)}$ are strictly less than $\tfrac32$.
Thus if $q_N$ or $q_{N-1}$ is odd, the set $B_2^{(N)}$ does not play any role for $\sup A_2$.
If $q_N$ and $q_{N-1}$ are even, then the set $B_2^{(N)}$ is empty.
It means that $\sup A_2 = \sup M_2$.
We can conclude that $\CR(\vv) = \sup(A_1 \cup A_2) =\sup(M_1\cup M_2
 \cup M_3)$.

\medskip

If the directive sequence equals $D^{a_1} G^{a_2} D^{a_3} G^{a_4} \cdots$, only the coordinates of the Parikh vectors of $r$ and $s$ are exchanged (see Remark \ref{Rem_GandDexchanged}).
\end{proof}

\section{CS Rote sequences with a small critical exponent} \label{S_smallexponent}

In this section, we present some corollaries of Theorem \ref{T_KritExpo}.
As we have mentioned, Currie, Mol, and Rampersad proved in \cite{CuMoRa} that there are exactly two  languages of rich binary sequences with the critical exponent less than $\tfrac{14}{5}$.
Both of them are languages of CS Rote sequences.

Let us recall  that the critical exponent depends only on the language of a  sequence and not on the sequence itself.
Therefore, there are infinitely many CS Rote sequences with the critical exponent less than $\tfrac{14}{5}$,  but all of them have one of two languages.
We show that among all languages of  CS Rote sequences only these two languages have the critical exponent less than $3$.
We also describe all  languages of CS Rote sequences with the critical exponent equal to $3$.

\begin{prop}  \label{Prop_MalyExponent}
Let $\vv$ be a CS Rote sequence  associated with  the standard Sturmian sequence $\uu = \S(\vv)$.
If $\emph{cr}(\vv) \leq  3$,  then the directive sequence of $\uu$ is of one of the following forms:
\begin{enumerate}
\item $G^{a_1} (D^{2}G^{2})^\omega$, where  $a_1 = 1$ or $a_1 = 3$; in this case $\emph{cr}(\vv) = 2+ \tfrac{1}{\sqrt{2}}$;
\item $G^{a_1} D^{4}(G^{2}D^2)^\omega$, where  $a_1 = 1$ or $a_1 = 3$; in this case $\emph{cr}(\vv) = 3$;
\item $G^{a_1}D^{1} G^{a_3}(D^{2}G^{2})^\omega$, where $a_1= 2$ or $a_1 = 4$ and $a_3 = 1$ or $a_3 = 3$; in this case $\emph{cr}(\vv) = 3$;
\item $D^{1}G^{a_2} (D^{2}G^{2})^\omega$, where $a_2= 1$ or $a_2 = 3$; in this case $\emph{cr}(\vv) = 3$.
\end{enumerate}
\end{prop}

\begin{proof}
For each $N \in \N$ we denote by $ \beta_N =a_{N+1} + 2 + \frac{q'_{N-1}-1}{q'_N}$ the number which is a candidate to join the set $M_1$.
We can easily compute that  $\beta_0 = a_1 + 2$, $\beta_1 = a_2 + 2$ and $\beta_N > 3$ for every $N \geq 2$.
Indeed, it suffices to realize that $q'_{-1} = q'_{0} = 1$, $q'_{1} = a_1 + 1 > 1$ and $(q'_N)_{N \geq 1}$ is an increasing sequence of integers, so $\frac{q'_{N-1}-1}{q'_{N}} \in (0,1)$ for all $N \geq 2$.
Since we look for a sequence $\vv$ with $\CR(\vv) \leq 3$, we have to ensure that $\beta_N \notin M_1$ for all $N \geq 2$ by the parity conditions.
It is also important to notice that $\sup M_3 \leq 3$. Indeed, since $\frac{q'_{N}-1}{q'_{N} + q'_{N-1}} \in \left[0,1\right)$ for all $N \in \N$, all elements of $M_3$ are less than $3$.

\medskip
First we assume that the directive sequence of $\uu$ is $G^{a_1} D^{a_2}G^{a_3}D^{a_4}\cdots$.
By Theorem \ref{T_KritExpo}, if $q_{N}$ is even, then $\beta_N \in M_1$, otherwise $\frac{1}{2}\beta_N \in M_2$.
To ensure $\sup M_1 \leq 3$, $q_N$ has to be odd for all $N \geq 2$.
Moreover, to ensure $\sup M_2 \leq 3$, $\beta_N \leq 6$ and so $a_{N+1} \leq 3$ for all $N \geq 2$.
Since $q_0=1$ is odd, $\frac{1}{2}\beta_0 = \frac{a_1}{2} + 1 \in M_2$ and so $a_1 \leq 4$.
We distinguish two cases.

(i) If $q_1 = a_1$ is odd, then $M_1$ is empty and $\frac{1}{2}\beta_1 = \frac{a_2}{2} +1 \in M_2$.
Thus $a_2 \leq 4$.
The recurrence relation $q_N = a_Nq_{N-1}+q_{N-2}$ with the odd initial conditions $q_0$ and $q_1$ produces $q_N$ odd for all $N \geq 2$ if and only if $a_N$ is even for all $N \geq 2$.
We can summarize that $a_1 \in \{1, 3\}$, $a_2 \in \{2, 4\}$ and $a_N = 2$ for all $N \geq 3$.

Let us observe that if $a_2 = 4$, then $\sup M_2 = 3$.
Since $\sup M_3 \leq 3$ and $M_1$ is empty, we conclude that $\CR (\vv) = \sup M_2 = 3$. It gives us Item 2 of our proposition.

If $a_2 = 2$, then it is easy to check that all elements of $M_2$ are smaller than $\tfrac52$ and thus $\sup M_2 \leq \tfrac52$.
Since $M_1$ is empty, to prove Item 1, it remains to deduce
$$\sup M_3= \sup \left\{2 + \frac{q'_{N}-1}{q'_{N-1} + q'_N} :  N\in \N\right\} = 2 + \frac{1}{\sqrt{2}} > \frac{5}{2}\,.$$
The sequence $(q'_N)_{N \geq 2}$ fulfils the recurrence relation $q'_N = 2q'_{N-1}+ q'_{N-2}$ with the initial conditions $q'_0 =1$ and $q'_1 = a_1+  1$.
This linear recurrence has the solution
\begin{equation}  \label{reseni}
q'_{N} = \tfrac{1}{2\sqrt{2}} \Bigl( ( a_1+\sqrt{2})(1+\sqrt{2})^N - ( a_1-\sqrt{2})(1-\sqrt{2})^N \Bigr)\,.
\end{equation}
Using the solution one can easily check that
${q'_{N-1} + q'_N} = {\sqrt{2}}\,q_N' +    c_N$, where $c_N=(a_1-\sqrt{2})(1-\sqrt{2})^N.$ As $\lim\limits_{N\to \infty} c_N=0$ and  $\lim\limits_{N\to \infty} q'_N=+\infty $,  we have
$$\lim_{n\to \infty} \Bigl(2 + \frac{q'_{N}-1}{q'_{N-1} + q'_N}\Bigr)=\lim_{n\to \infty} \Bigl(2 + \frac{q'_{N}-1}{\sqrt{2} q'_N +c_N}\Bigr)=
2+\frac{1}{\sqrt{2}}\,.$$
Moreover, since $c_N>-1$, all elements of $M_3$ are smaller than the limit $2+\frac{1}{\sqrt{2}}$. Indeed,
$$2 + \frac{q'_{N}-1}{ {\sqrt{2}}\,q_N' + c_N}  < 2 + \frac{q'_{N}-1}{ {\sqrt{2}}\,q_N' -1} < 2 + \frac{1}{\sqrt{2}}\,. $$

%
%

(ii) If $q_1 = a_1$ is even, then $M_1 = \{\beta_1 = a_2 + 2\}$ since all the other $q_N$ are odd.
Thus $a_2 = 1$ and $\CR(\vv) \geq 3$.
The recurrence relation $q_N = a_Nq_{N-1} + q_{N-2}$ with the initial conditions $q_0 = 1$, $q_1 = a_1 \in \{2,4\}$ produces odd $q_N$ for all $N \geq 2$ if and only if $a_3$ is odd and $a_N$ is even for all $N \geq 4$.
Moreover, to ensure $\sup M_2 \leq 3$, we have to take $a_1 \leq 4$ and $a_{N} \leq 3$ for every $N \geq 3$.
Together with the parity conditions we get $a_1 \in \{2,4\}$, $a_2 = 1$, $a_3 \in \{1,3\}$, and $a_N = 2$ for all $N \geq 4$.
Since $\sup M_3 \leq 3$, the set $M_3$ can be omitted.
Together it means that in this case $\CR(\vv) = 3$ and it corresponds to Item 3 of our statement.

\medskip
Now we assume that the directive sequence of $\uu$ is $D^{a_1}G^{a_2} D^{a_3}G^{a_4}\cdots$.
By Theorem \ref{T_KritExpo}, if $p_{N}$ is even, then $\beta_N \in M_1$, otherwise $\frac{\beta_N}{2} \in M_2$.
Since $p_0=0$ is even, $\beta_0 = a_1+2 \in M_1$.
Therefore, $a_1 = 1$ and $ \CR(\vv) \geq 3$.
Since $p_1 = a_1p_0 + p_{-1} = 1$ is odd, $\frac{\beta_1}{2} = \frac{a_2}{2} + 1 \in M_2$ and so $a_2 \leq 4$.
To guarantee $\CR(\vv) \leq 3$, $p_N$ has to be odd and $a_{N+1} \leq 3$ for all $N \geq 2$.
The recurrence relation $p_N = a_Np_{N-1} + p_{N-2}$ with the initial conditions $p_0=0$ and $p_1 = 1$ produces $p_N$ odd for all $N \geq 2$ if and only if $a_2$ is odd and $a_N$ is even for all $N\geq 2$.
Clearly, the set $M_3$ can be omitted.
We may conclude that $\CR(\vv) = 3$ and $a_1 = 1$, $a_2 \in \{1,3\}$, and $a_N = 2$ for all $N \geq 3$, which corresponds to Item 4.

\end{proof}

\begin{rem}
Let us emphasize that all standard Sturmian sequences from Proposition \ref{Prop_MalyExponent}, i.e., which are associated with CS Rote sequences with the critical exponent $\leq 3$, are morphic images of the fixed point of the morphism $D^2G^2$.
If follows directly from the fact that their directive sequences have the periodic suffix $(D^2G^2)^\omega$.
\end{rem}

\begin{exam}\label{Ex_Rozdil_cr_u_Eu}
In Proposition~\ref{Prop_MalyExponent}, we have shown that the CS Rote sequence $\vv$ such that $\S(\vv)$ has the directive sequence $G(D^2 G^2)^{\omega}$ has the critical exponent $\CR(\vv)=2+\frac{1}{\sqrt{2}}$.
Let us determine the critical exponent $\CR(\vv')$ of the CS Rote sequence $\vv'$ associated to the standard Sturmian sequence $\S(\vv')$ obtained by the exchange of letters from $\S(\vv)$, i.e., $\S(\vv')=E(\S(\vv))$.

By Remark \ref{Rem_Zamena}, the directive sequence of $\S(\vv')$ equals $D(G^2D^2)^{\omega}$.
Thus we have $\theta=[0,1,2,2,2,\dots]$ and it is readily seen that $p_N$ is even if and only if $N$ is even.
Let us calculate $\CR(\vv')$ by Theorem~\ref{T_KritExpo}. We have
$$M_1=\{a_{2N+1} + 2 + \frac{q'_{2N-1}-1}{q'_{2N}}: N \in \N\}=\{3\}\cup \{4 + \frac{q'_{2N-1}-1}{q'_{2N}}: N \in \N, N \geq 1\}\,.$$
Using equation \eqref{reseni}, we can check that the sequence $\bigl(\frac{q'_{2N-1}-1}{q'_{2N}}\bigr)$ is increasing and has the limit $\frac{1}{1 + \sqrt{2}}$, and therefore $\sup M_1=4+\frac{1}{1+\sqrt{2}}$.
Since the elements of $M_2$ and $M_3$ are $\leq 3$, we can conclude that $\CR(\vv')=4+\frac{1}{1+\sqrt{2}}$.
\end{exam}

It is well-known that among Sturmian sequences the sequence with the lowest possible critical exponent is the Fibonacci sequence ${\bf f}$, which has $\CR({\bf f})  =  3+ \tfrac{2}{1+\sqrt{5}} \sim 3. 602$.
The following theorem implies that there are uncountably many CS Rote sequences with the critical exponent smaller than $\CR({\bf f})$.

\begin{thm}\label{Thm_Exponent72}
Let $G^{a_1} D^{a_2}G^{a_3}D^{a_4}\cdots $  be the directive sequence of a standard Sturmian sequence  $\uu$ and let $\vv$ be the CS Rote sequence associated with $\uu$.
Then $\emph{cr}(\vv) < \tfrac72$ if and only if the sequence $a_1a_2 a_3 \cdots$  is a concatenation of the blocks from the following list:

$L_0$: \ \ $111$;

$L_1$: \ \  $s1$, where $s \in\{2,4\}$;

$L_2$: \ \ $cs31$, where $ c \in \{1,3\}$  and $s \in\{2,4\}^*$;

$L_3$: \ \ $c{\bf s}$, where $ c \in \{1,3\}$  and ${\bf s} \in \{2,4\}^\mathbb{N}$,

\noindent and if the block $L_0$ appears in $a_1a_2a_3\cdots$, then it is a prefix of $a_1a_2a_3 \cdots$.
\end{thm}

\begin{proof}
As in the proof of Proposition \ref{Prop_MalyExponent}, we again use the sets $M_1$ and $M_2$ from Theorem \ref{T_KritExpo}.
The set $M_3$ can be omitted since $\sup M_3\leq 3$.

Let us recall that $q_0 =1$, $q_1 = a_1$, and if $q_N$ is even, then $\beta_N = a_{N+1} + 2 + \frac{q'_{N-1}-1}{q'_N} \in M_1$, otherwise $\frac{1}{2}\beta_N \in M_2$.
First we suppose that $\CR(\vv) < \tfrac72$ and we deduce several auxiliary observations for each $N \in \N$:

\begin{enumerate}
\item If $q_N$ even, then $a_{N+1} = 1$.   \quad

Proof: It follows from the  inequality  $ \beta_N  < \tfrac72$.

\item If $q_N$ odd, then $a_{N+1} \in \{1,2,3,4\}$.   \quad

Proof: It follows from the  inequality  $ \tfrac12\beta_N  < \tfrac72$.

\item If $q_{N-1}$ odd and $q_{N}$ even, then $q_{N+1}$ odd.  \quad

Proof: It follows from Item 1 and the relation $q_{N+1} = a_{N+1}q_{N}+q_{N-1} = q_N + q_{N-1}$.

\item If $q_{N}$ even, $q_{N+1}$ odd, and $q_{N+2}$ even, then $a_{N+2} \in \{2,4\}$.  \quad

  Proof: It follows from Item 2 and the relation $q_{N+2} = a_{N+2}q_{N+1}+q_{N}$.

\item If $q_{N}$ even, $q_{N+1}$ odd, and $q_{N+2}$ odd, then $a_{N+2} \in \{1,3\}$.  \quad

 Proof: It follows from Item 2 and the relation $q_{N+2} = a_{N+2}q_{N+1}+q_{N}$.

\item If $q_{N}$ odd, $q_{N+1}$ odd, and $q_{N+2}$ odd, then $a_{N+2} \in \{2,4\}$.  \quad

Proof: It follows from Item 2 and the relation $q_{N+2} = a_{N+2}q_{N+1}+q_{N}$.

\item If $q_{N}$ odd, $q_{N+1}$ odd, and $q_{N+2}$ even, then $a_{N+2}=\{1,3\}$. Moreover, if $a_{N +2}  = 1$, then $N = 0$ and $a_1 = 1$. \quad

Proof: Item 2 and the relation $q_{N+2} = a_{N+2}q_{N+1}+q_{N}$ imply $a_{N+2} \in \{1,3\}$.
Assume that $a_{N+2} = 1$.
By Item 1, $a_{N+3} = 1$.
As $q_{N+2}$  is even,  $\beta_{N+2} \in M_1$ and so $\beta_{N+2} = 3 +  \frac{q'_{N+1}-1}{q'_{N+2}} < \tfrac{7}{2}$.
By some simple rearrangements and applications of the recurrence relation, we can rewrite this inequality equivalently as $(a_{N+1} -1)q'_{N} + q'_{N-1} < 2$.
It is easy to verify that this inequality holds only for $N = 0$ and $a_1 = 1$ or $N = 1$ and $a_2 = 1$.
Nevertheless, the second case leads to a contradiction with the assumption that both $q_1,  q_2 = a_2q_1 +1$ are odd.

\item  Let  $M>  N+2$.   If  $q_{N}$ even,   $q_{M}$ even, and $q_{K}$ odd for all $K, N<K <M$, then $a_{N+2} \in \{1,3\}$, $a_M = 3$ and   $a_{K} \in \{2, 4\}$ for all $K,  N+2 <K < M$.

Proof:  Item 6 implies $a_{K} \in \{2, 4\}$ for all $K$, $N+2 <K < M$. By Item 5, $a_{N+2} \in \{1,3\}$ and by Item 7, $a_M =3$.
\end{enumerate}
\medskip

Using the previous claims we show that for each $J$ for which $q_J$ is even, at the position $J+1$ ends one of the blocks $L_0$, $L_1$, or $L_2$.
Moreover, the block $L_0$ can only occur as a prefix of the sequence $a_1a_2 a_3 \cdots$, while each of the blocks $L_1$ and $L_2$ is either a prefix of $a_1a_2 a_3 \cdots$ or it starts at the position $I+2$, where $I$ is the greatest integer smaller than $J$ for which $q_I$ is even.
We discuss three cases:

\begin{itemize}
\item Let $q_1$ be even.
Then Item 1 implies $a_2 = 1$.
And since $a_1 = q_1$ is even and $q_0 = 1$ is odd, by Item 2, we get $a_1\in \{ 2, 4\}$.
Thus the prefix $a_1a_2$ of the sequence $a_1a_2a_3 \cdots$ is of the form $L_1$ from our list.

\item Let $J>1$ be the first index such that $q_J$ is even.
As $q_1 = a_1$, $a_1$ is odd, and by Item 2, we get $a_1 \in \{1,3\}$.
Item 6 implies $a_K \in \{ 2,4\}$ for all $K, 1<K<J$.
By Item 7, $a_J = 3$ or $a_J = 1$.
But if $a_J = 1$, then $J = 2$ and $a_1 = 1$.
Finally, Item 1 implies $a_{J+1} = 1$.
Thus the prefix  $a_1a_2\cdots  a_Ja_{J+1}$ of the sequence $a_1a_2a_3 \cdots$ is of the form $L_0$ or $L_2$ from our list.

\item Let $I$ be an index such that $q_I$ is even and let $J$ be the smallest index greater than $I$ for which $q_J$ is even.
By Item 1, $a_{I+1} = 1$.
The word $a_{I+2} \cdots a_Ja_{J+1}$ is by Item 4 or Item 8 either of the form $L_1$  or $L_2$.
\end{itemize}

If there are infinitely many even denominators $q_N$, then we have shown that the sequence $a_1a_2a_3 \cdots$ is concatenated from the blocks $L_0$, $L_1$, and $L_2$ ($L_0$ can only be a prefix).
It remains to consider the case when only finitely many denominators $q_N$ are even.

\begin{itemize}
\item Let $q_N$ be odd for every $N \in \N$.
Then $q_1 = a_1$ is odd.
Especially, since both $q_0$ and $q_1 = a_1$ are odd, Item 2 implies $a_ 1 \in\{1, 3\}$ and it follows from Item 6 that $a_2a_3a_4\cdots \in \{2,4\}^\mathbb{N}$.
Therefore, the sequence $a_1a_2a_3 \cdots$ is equal to the block $L_3$ from our list.

\item Let $L \geq 1$ be the last index such that $q_L$ is even.
In particular, it means that $q_N$ is even only for a finite number of indices.
Item 1 implies that $a_{L+1} = 1$.
By Items 5 and 6, the suffix $a_{L+2}a_{L+3}a_{L+4}\cdots$ of the sequence $a_1a_2a_3 \cdots$ equals to $L_3$.
\end{itemize}

\medskip

Now we have to show that any directive sequence $G^{a_1} D^{a_2}G^{a_3}D^{a_4}\cdots$ such that $a_1a_2a_3 \cdots$ is concatenated of the blocks from the list gives a standard Sturmian sequence $\uu$ such that the CS Rote sequence $\vv$ associated to $\uu$ has the critical exponent less than $\tfrac72$.

If $q_N$ is odd, then $\tfrac12 \beta_N  = \tfrac12 a_{N+1} + 1+\tfrac{q'_{N-1} -1}{2q'_{N}}  < 3 + \tfrac{q'_{N-1}}{2q'_{N-1} + q'_{N-2}}< \tfrac72$, as each $a_{N+1}$ is $\leq 4$.

If $q_N$ is even, it is easy to prove by induction on $N$ that there is a block of the form $L_0$, $L_1$, or $L_2$ ending at the position $N+1$ (and $L_0$ only for $N = 2$).
In particular, it means that $a_{N+1}=1$ and $a_N \geq 2$ or $a_1 = a_2 = a_3 = 1$.
In the first case, we get $ \beta_N  \leq   3+ \tfrac{q'_{N-1}}{2q'_{N-1} + q'_{N-2}}< \tfrac72$, while in the second case, we get $\beta_2 = 3 + \frac{a_1 + 1 - 1}{a_2(a_1 + 1) + 1} = 3 + \frac{1}{3} < \tfrac72$.

To show that $\sup (M_1\cup M_2)< \tfrac72$, we need to show that $\sup \tfrac{q'_{N-1}}{2q'_{N-1} + q'_{N-2}} < \tfrac12$.  As $a_{N} \leq 4$ for all  $N$, we can estimate
 $\tfrac{q'_{N-2}}{q'_{N-1}} \geq \tfrac{q'_{N-2}}{4q'_{N-2} +q'_{N-2}  }  = \tfrac15$ and thus $ \tfrac{q'_{N-1}}{2q'_{N-1} + q'_{N-2}}  = \tfrac{1}{2 +\tfrac{q'_{N-2}}{q'_{N-1}}  } \leq \tfrac5{11}$.
\end{proof}

\begin{rem}
It would be interesting to reveal some topological properties of the set $$\CR_{Rote} := \{\CR(\vv) : \vv  \text { is a CS Rote  sequence}\},$$ for instance, to find its accumulation points in the interval $(3, \tfrac72)$.
The proof of the previous theorem implies that there is no CS Rote sequence with the critical exponent between $3+\tfrac5{11}$ and $ 3+\tfrac12$.
We even believe that for any CS Rote sequence $\vv$, the following implication holds: If $\CR(\vv)<3+\frac{1}{2}$, then $\CR(\vv)<3+\frac{1}{1+\sqrt{3}}$.
\end{rem}

\section{The recurrence function of CS Rote sequences} \label{S_recurrencefunction}
The main result of this section is Theorem~\ref{Thm_RecFunctionRoteDetailed}, where we describe the recurrence function of any CS Rote sequence in terms of the convergents related to the associated Sturmian sequence.
To obtain this result, we proceed similarly as in the previous parts concerning the critical exponent, i.e., we transform our task of finding the recurrence function of a CS Rote sequence into studying some properties of its associated Sturmian sequence.
Let us emphasize that we may still restrict our consideration to CS Rote sequences associated with standard Sturmian sequences without loss of generality.

\begin{defi}
Let $\uu$ be a uniformly recurrent sequence. The mapping $R_\uu: \N \to \N$ defined by
$$R_\uu(n)=\min \{N \in \N : \text{each factor of $\uu$ of length $N$ contains all factors  of $\uu$ of length $n$}\}$$
is called the \emph{recurrence function} of $\uu$.
\end{defi}

The definition of the recurrence function may be reformulated in terms of return words \cite{CaCh}.

\begin{obs}\label{Obs_RnRetWords}
Let $\uu$ be a uniformly recurrent sequence. Then
$$R_\uu(n)=\max \{|r|\in \N : \text{$r$ is a return word to a factor of $\uu$ of length $n$}\}+n-1.$$
\end{obs}

Moreover, to determine $R_\uu(n)$ we can restrict our consideration to return words to bispecial factors of~$\uu$.

\begin{lem}\label{Lem_BN}
Let $\uu$ be a uniformly recurrent aperiodic sequence. For $n \in \N$, we denote
$$\mathcal{B}_{\uu}(n)= \{ b\in \mathcal{L}(\uu) : \exists w \in \mathcal{L}(\uu), |w| = n, \   \text{such that $b$ is the shortest bispecial factor containing $w$} \}\,.
$$
Then $$R_\uu(n)=\max \{|r|: \text{$r$ is a return word to  $b \in \mathcal{B}_{\uu}(n)$}\}+n-1.$$
\end{lem}

\begin{proof}
For evaluation of  $R_\uu(n)$  we use Observation \ref{Obs_RnRetWords}.  Let $w\in  \mathcal{L}(\uu)$ and $ |w| = n $.

If $w$ is not right special, then there exists a unique letter $x$ such that  $wx \in \mathcal{L}(\uu)$.
Obviously, the occurrences of $w$ and $wx$ in $\uu$ coincide. Therefore, return words to $w$ and $wx$ coincide as well.

If $y$ is not left special, then there is a unique letter $y$ such that $yw\in \mathcal{L}(\uu)$. If $r$ is a return word to $w$, then the word $yry^{-1}$ is a return word to $yw$ and the return words $r$ and $yry^{-1}$ are of the same length.

These two facts imply that the lengths of return words to $w$ equal the lengths of return words to the shortest bispecial factor containing $w$.
\end{proof}

The following lemma shows that for a CS Rote sequence $\vv$ associated with the Sturmian sequence $\uu$ the sets $\mathcal{B}_\vv(n+1)$ and $\mathcal{B}_\uu(n)$ correspond naturally for every $n \in \N$. Thus to determine the set $\mathcal{B}_\vv(n+1)$, we first describe the set $\mathcal{B}_\uu(n)$ for a standard Sturmian sequence $\uu$.

\begin{lem}\label{Lem_ShortestBSinRoteSturm}
Let $w$ be a~factor of length $n+1$ in a CS Rote sequence $\vv$ and let $v$ be the shortest bispecial factor of $\vv$ containing $w$. Then the factor $\S(v)$ is the shortest bispecial factor of the associated Sturmian sequence $\S(\vv)$ such that $\S(v)$ contains $\S(w)$.
\end{lem}

\begin{proof}
The statement is a consequence of the simple fact that $\S(v)$ is a bispecial factor of $\S(\vv)$ if and only if $v$ and $E(v)$  are  bispecial factors of $\vv$. (See Lemma~\ref{lem_Bispecial}.)
\end{proof}

In the sequel, we will essentially use a characterization of Sturmian sequences by palindromes from~\cite{DrPi}.
Let us first recall some basic notions. Consider an alphabet $\A$. The assignment $w=w_0w_1\cdots w_{n-1} \to \overline{w} =
w_{n-1}w_{n-2}\cdots w_{0}$ is called a~\textit{mirror mapping}, and
the word $\overline{w}$ is called the \textit{reversal} or the \textit{mirror image} of $w \in \A^*$.
A~word $w$ which coincides with its mirror image $\overline{w}$ is a~\textit{palindrome}.
If $p$ is a palindrome of odd length, then the \textit{center} of $p$ is a letter $a$ such that $p = sa \overline{s}$ for some $s \in \A^*$. The center of a palindrome $p$ of even length is the empty word $\varepsilon$.

\begin{thm}[\cite{DrPi}]\label{Thm_DrPi}
A sequence $\uu$ is Sturmian if and only if $\uu$ contains one palindrome of every even length and two palindromes of every odd length.
\end{thm}

Moreover, when studying in detail the proof of this theorem presented by Droubay and Pirillo ~\cite{DrPi}, we deduce that any two palindromes of the same odd length have distinct centers, one has the center $0$ and the other one has the center $1$.
In fact, we get the following corollary.

\begin{coro}\label{Coro_OnePalExt}
A~binary sequence $\uu$ is Sturmian if and only if every palindrome in ${\mathcal L}(\uu)$ has a unique palindromic extension, i.e., for any palindrome $p \in \mathcal{L}(\uu)$
there exists a~unique letter $a\in \{0,1\}$ such that $apa \in
\mathcal{L}(\uu)$.
\end{coro}

We believe that Theorem~\ref{Thm_shortestBS} is already known. However, since we have not found it in the literature, we add its proof.
For this purpose, we need an auxiliary lemma.
Let us recall that the language $\mathcal{L}(\mathbf{u})$ of a Sturmian sequence $\uu$ is \textit{closed under reversal}, i.e., $\mathcal{L}(\mathbf{u})$ contains with every factor $w$ also its reversal $\overline{w}$, and all bispecial factors of $\uu$ are palindromes.

\begin{lem}\label{Lem_CenterOfBS}
Let $\uu$ be a Sturmian sequence. Let $p \in \mathcal{L}(\uu)$ be a palindrome and let $v$ be the shortest bispecial factor containing $p$.
\begin{enumerate}
\item Then $p$ is a central factor of $v$, i.e., $v = sp\overline{s}$ for some word $s$.
\item If $v'$ is the shortest bispecial factor with the same center as $p$ and  $|v'|\geq |p|$, then $v'=v$.
\end{enumerate}
\end{lem}

\begin{proof}
\begin{enumerate}
\item
Let  $v = sp$ be the shortest left special factor containing $p$, in particular $0sp$ and $1sp$  belong to ${\mathcal L}(\uu)$.  Since the language $\mathcal{L}(\uu)$ is closed under reversal, $p\overline{s}$ is right special, i.e., $p\overline{s}0$ and $p\overline{s}1$ belong to ${\mathcal L}(\uu)$.   As $s$ is the only possible extension of $p$ to the left by a factor of length $|s|$, both $sp\overline{s}0$ and $sp\overline{s}1$ belong to ${\mathcal L}(\uu)$.   By the same argument,  $\overline{s}$ is the only possible extension of $p$ to the right by a factor of length $|s|$. Therefore,   $0sp\overline{s}$ and $1sp\overline{s}$ belong to ${\mathcal L}(\uu)$. Thus $sp\overline{s}$ is the shortest bispecial factor containing $p$.

\item Assume for contradiction that $v\not =v'$. Since $v$ and $v'$ are palindromes with the same centers, there exists a palindrome $q$ such that $v'=s'q\overline{s'}$ and $v=sq\overline{s}$. Let $q$ be the longest palindrome with this property. If $|q|=|v'|$, then necessarily $v'=v$. If $|v'|>|q|$, then the last letters of $s'$ and $s$ are distinct and $q$ is a palindrome with two distinct palindromic extensions. This contradicts Corollary~\ref{Coro_OnePalExt}.
\end{enumerate}
\end{proof}

\begin{thm}\label{Thm_shortestBS}
Let $\uu$ be a Sturmian sequence and $n \in \N, n \geq 1$. Find the shortest bispecial factors $P_\varepsilon$, resp. $P_0$, resp. $P_1$ of length greater than or equal to $n$ with the center $\varepsilon$, resp. $0$, resp. $1$. Then the following statements hold:
\begin{enumerate}
\item Let $w$ be a factor of $\uu$ of length $n$ and let $v$ be the shortest bispecial factor containing $w$. Then $v \in \{P_\varepsilon, P_0, P_1\}$.
\item If $\uu$ contains no bispecial factor of length $n-1$, then for each $i\in \{\varepsilon, 0,1\}$ there exists a factor $w$ of $\uu$ of length $n$ such that the shortest bispecial factor containing $w$ is $P_i$.
\item If there exists a bispecial factor $v$ of $\uu$ of length $n-1$ and let $i \in \{\varepsilon, 0,1\}$ be the center of the palindrome $v$, then for each $j\in \{\varepsilon, 0,1\}, \ j\not =i,$ there exists a factor $w$ of $\uu$ of length $n$ such that the shortest bispecial factor containing $w$ is $P_j$, while $P_i$ is not the shortest bispecial factor containing $w$ for any factor $w$ of $\uu$ of length $n$.
\end{enumerate}
\end{thm}

\begin{proof}
Consider the {\em Rauzy graph} of $\uu$ of order $n-1$, i.e., a~directed graph $\Gamma_{n-1}$ whose vertices are factors of $\uu$ of length $n-1$
and edges are factors of $\uu$ of length $n$. An edge $e$ starts
in the vertex $x$ and ends in the vertex $y$ if $x$ is a prefix
and $y$ is a suffix of $e$.
Denote $\ell$, resp. $r$ the vertex corresponding to the unique left special, resp. right special factor of length $n-1$. Furthermore, we denote  the shortest path from $\ell$ to $r$ by $p_A$, and the shortest paths  of non-zero length starting in $r$ and ending in $\ell$ by $p_B$ and $p_C$. If $\uu$ has no~bispecial factor of length $n-1$, then $p_A$ has  a~positive length,  see Figure~\ref{Fig_Rauzy}(a). If $\uu$ has a~bispecial factor of length $n-1$, then the path $p_A$ consists of a~unique vertex -- the bispecial factor $b$, see Figure~\ref{Fig_Rauzy}(b).
\begin{figure}[!h]
\begin{center}
\resizebox{11 cm}{!}{\includegraphics{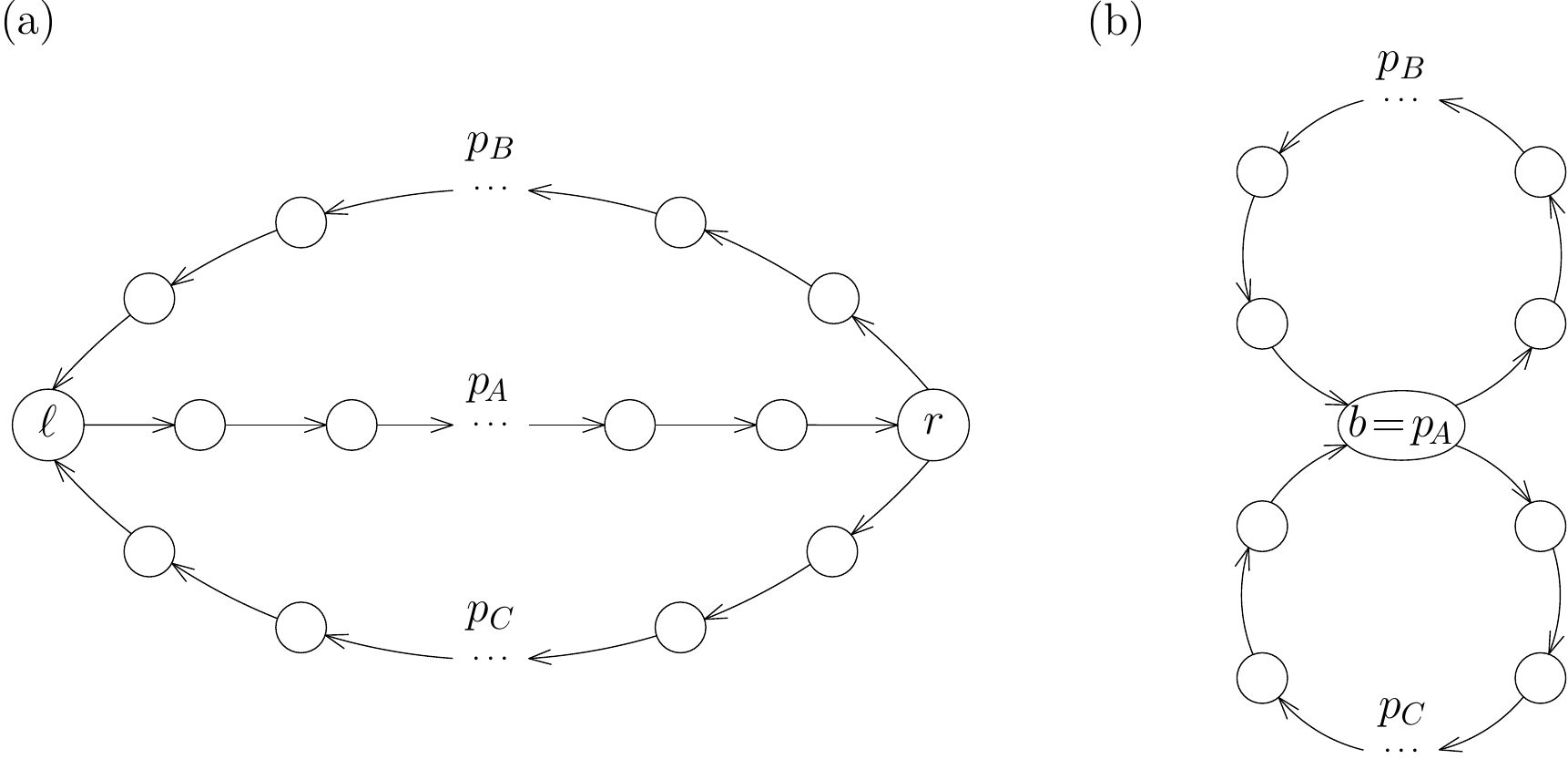}}
\end{center}
\caption{The Rauzy graph of a~Sturmian word (a) without a bispecial vertex, (b) with a bispecial vertex.} \label{Fig_Rauzy}
\end{figure}

Observing these Rauzy graphs, it is obvious that for each edge $e$ from the path $p_x$, where $x \in \{A,B,C\}$, the shortest bispecial factor containing $e$ is the same as the shortest bispecial factor containing $p_x$.
Clearly, if $p_A$ has length $0$, then $x \in \{B,C\}$.

Since the language ${\mathcal L}(\uu)$ is closed under reversal, the mirror mapping restricted to the factors of length $n-1$ and $n$ is an automorphism of the graph $\Gamma_{n-1}$.
Let us suppose that  a palindrome $q$ of length $n-1$ or $n$ is contained in $p_x$ as an edge or an inner vertex.
As any palindrome is mapped onto itself, $\overline{r}=\ell$, and $\overline{\ell}=r$, this path $p_x$ is mapped onto itself, i.e., $\overline{p_x}=p_x$, and the palindrome $q$ is a central factor of $p_x$.
On one hand, it means that $p_x$ is a palindrome with the same center as $q$, on the other hand, it also means that $p_x$ cannot contain any other palindrome of length $n-1$ and $n$ as its edge or an inner vertex.

By Theorem~\ref{Thm_DrPi} and  the  comment  after it, there  are exactly three palindromes among all vertices and edges of $\Gamma_{n-1}$ (all factors of length $n-1$ or $n$), and moreover, they have distinct centers.
We may conclude that the paths $p_A$, $p_B$, $p_C$ are palindromes and their centers are distinct. The rest of the proof follows from Lemma~\ref{Lem_CenterOfBS}.

\end{proof}

\begin{obs}\label{Obs_ParikhPal}
Let $P$ be a palindrome. Then its Parikh vector satisfies:
\begin{enumerate}
\item $\vec{V}(P)= \binom{0}{0} \mod 2$ if and only if $P$ has the center $\varepsilon$;
\medskip

\item $\vec{V}(P)=\binom{1}{0} \mod 2$ if and only if $P$ has the center $0$;
\medskip

\item $\vec{V}(P)=\binom{0}{1} \mod 2$ if and only if $P$ has the center $1$.
\end{enumerate}
\end{obs}

Let us recall that a factor of a standard Sturmian sequence $\uu$ is bispecial if and only if it is a palindromic prefix of $\uu$.
Therefore we can order the bispecial factors of a given standard Sturmian sequence $\uu$ according to their lengths. We denote the $k$-th bispecial factor of $\uu$ by $BS(k)$.
Thus $BS(0)=\varepsilon$, $BS(1)=a$, where $a$ is the first (and the more frequent) letter of $\uu$ etc.

\medskip

The sequences $(p_N)$,  $(q_N)$,  and  $(q'_N)$  we use in the remaining part of the paper were  introduced in Notation \ref{Theta}, the notation $\mathcal{B}_{\uu}(n)$ comes from Lemma   \ref{Lem_BN}.

\begin{thm}\label{Thm_P01E}
Let $\uu$ be a standard Sturmian sequence with the directive sequence $G^{a_1}D^{a_2}G^{a_3}D^{a_4} \cdots $ or $D^{a_1}G^{a_2}D^{a_3}G^{a_4} \cdots $, and $n \in [q_N', q_{N+1}')$ for some $N \in \N$.
Put $M = a_0+a_1+a_2+ \cdots +a_{N}$, where $a_0=0$.
\begin{itemize}
\item  If $n  \in [q_N', q_{N+1}'-1)$ and $n-1$ is not the length of a bispecial factor, then
$$ \mathcal{B}_{\uu}(n) = \{BS(M+m), BS(M+m+1), BS(M+a_{N+1}+1)\} \ \text{for some $m \in \{0,\dots,a_{N+1}-1\}$}.$$
\item If $n  \in [q_N', q_{N+1}'-1]$ and $n-1$ is the length of a bispecial factor, then
$$\mathcal{B}_{\uu}(n) = \{BS(M+m), BS(M+a_{N+1}+1)\}\ \text{ for some $m \in \{0,\dots,a_{N+1}\}$}.$$
\end{itemize}
\end{thm}

\begin{proof}
Assume that $\uu$ has the directive sequence $G^{a_1}D^{a_2}G^{a_3}D^{a_4} \cdots $.
By Proposition~\ref{ParikhRSB}, the Parikh vectors of bispecial factors satisfy
\begin{equation}\label{alternace}
\vec{V}\bigl(BS(M+i+1)\bigr) =  \vec{V}\bigl(BS(M+i)\bigr) + \left(\!\!\begin{array}{c}p_N\\ q_N\end{array}\!\!\right) \quad \text{for } \ i=0,1,\ldots, a_{N+1} -1
\end{equation}
\begin{equation}\label{jiny}
\text{and} \qquad \vec{V}\bigl(BS(M+a_{N+1}+1)\bigr) =  \vec{V}\bigl(BS(M+a_{N+1})\bigr) + \left(\!\!\begin{array}{c}p_{N+1}\\ q_{N+1}\end{array}\!\!\right).
\end{equation}
Using Observation~\ref{Obs_ParikhPal} and the relation $\left(\!\!\begin{array}{c}p_N\\ q_N\end{array}\!\!\right)\not =\left(\!\!\begin{array}{c}0\\ 0\end{array}\!\!\right) \mod 2$ from  Lemma~\ref{Lem_Convergents}, we deduce that the centers of palindromes $BS(M+i)$, where $i=0,1, \ldots, a_{N+1}$, alternate between two distinct elements of $\{\varepsilon, 0,1\}$.   The third element of  $\{\varepsilon, 0,1\}$ is the center of the palindrome $BS(M+a_{N+1}+1)$, as
$\left(\!\!\begin{array}{c}p_N\\ q_N\end{array}\!\!\right)\not =\left(\!\!\begin{array}{c}p_{N+1}\\ q_{N+1}\end{array}\!\!\right) \mod 2$, see Lemma~\ref{Lem_Convergents}.  By Proposition~\ref{ParikhRSB}, the length of $BS(M-1)$ equals $q_N'-2$ and the length of $BS(M+a_{N+1}-1)$ equals $q_{N+1}'-2$.
\begin{itemize}
\item Let  us discuss the case  $n =   q_{N+1}'-1$.   The palindromes   $BS(M+a_{N+1}-1)$, $BS(M+a_{N+1})$ and $BS(M+a_{N+1}+1)$ have distinct centers, and   $n-1$  is the length of the  palindrome $BS(M+a_{N+1}-1)$.  Item 3 of Theorem \ref{Thm_shortestBS} implies that all factors of length $n$ occur in
$BS(M+a_{N+1})$ and $BS(M+a_{N+1}+1)$.   Therefore, the set $\mathcal{B}_\uu(n)$ consists of these two bispecial palindromes.

\item Now we assume that $n \in [q_{N}',  q_{N+1}'-2]$. Clearly, the length of $BS(M-1)$  is strictly smaller then $n-1$  and  $n$ does not exceed the length of  $BS(M+a_{N+1}-1)$. We choose the smallest  $m \in \{0,1,\dots,a_{N+1}-1\}$ such that  $n\leq |BS(M+m)|$.  The bispecial factors $ BS(M+m)$,  $ BS(M+m+1)$, and $BS(M+a_{N+1}+1)$ have distinct centers.

If $n-1$ is not the length of any bispecial factor, then Item 2 of Theorem \ref{Thm_shortestBS} implies that $ \mathcal{B}_\uu(n) = \{BS(M+m), BS(M+m+1), BS(M+a_{N+1}+1)\}$.

If $n-1$ is the length of a bispecial factor, then Item 3 of Theorem \ref{Thm_shortestBS} together with the fact that the centers of $BS(M+i)$ alternate for $i=0,1,\ldots, a_{N+1} -1$ implies that $ \mathcal{B}_\uu(n) = \{BS(M+m), BS(M+a_{N+1}+1)\}$.
\end{itemize}
If $\uu$ has the directive sequence $D^{a_1}G^{a_2}D^{a_3}G^{a_4} \cdots $, then the proof will be analogous, only the coordinates of the Parikh vectors will be exchanged, see Remark~\ref{Rem_GandDexchanged}.
\end{proof}

\begin{rem}
The recurrence function of a standard Sturmian sequence $\uu$ with the directive sequence $G^{a_1}D^{a_2}G^{a_3}D^{a_4} \cdots $ or $D^{a_1}G^{a_2}D^{a_3}G^{a_4} \cdots $ is known to satisfy $R_\uu(n)=q_{N+1}'+q_N'+n-1$ for every $n \in [q_N', q_{N+1}')$.
Let us show that this  formula is a consequence of the previous statements.  Indeed, by  Lemma \ref{Lem_BN}, we have to  find the longest return word to the bispecial factor from the set $\mathcal{B}_{\uu}(n)$ described in Theorem \ref{Thm_P01E}.
Using Proposition~\ref{ParikhRSB}, we find that the longest one is the return word  $s$ corresponding to the bispecial factor $BS(M+a_{N+1}+1)$.
Its length is $|s| = q_{N+1}'+q_N'$.
And thus Lemma \ref{Lem_BN} implies the above mentioned formula, which was  obtained by Hedlund and Morse already in 1940, see \cite{MoHe}.
\end{rem}

We have prepared everything we need to derive the formula for the recurrence function of CS Rote sequences.
For this purpose, we recall Theorem 3.10 from~\cite{MePeVu}:

\begin{thm}\label{Thm_RetWordsRote}
Let $\vv$ be a CS Rote sequence associated with the  standard Sturmian sequence $\uu=\S(\vv)$. Let $v$ be a non-empty prefix of $\vv$ and $u=\S(v)$.
Let $r$, resp. $s$ be the more frequent, resp. the less frequent return word to $u$ in $\uu$ and let $\ell$ be a positive integer such that $\uu$ is a concatenation of the blocks $r^{\ell }s$ and $r^{\ell+1}s$. Then the prefix $v$ of $\vv$ has three return words $A, B, C$ satisfying:
\begin{enumerate}
\item If $r$ is stable and $s$ unstable, then $\S(A0)=r, \quad \S(B0)=sr^{\ell}s, \quad \S(C0)=sr^{\ell+1}s$.
\item If $r$ is unstable and $s$ stable, then $\S(A0)=s, \quad \S(B0)=rr, \quad \S(C0)=rsr$.
\item If both $r$ and $s$ are unstable, then $\S(A0)=rr, \quad \S(B0)=rs, \quad \S(C0)=sr$.
\end{enumerate}
\end{thm}

We will use the previous theorem for the determination of return words to bispecial factors of CS Rote sequences (which are by Lemma \ref{Lem_ShortestBSinRoteSturm} associated to bispecial factors of Sturmian sequences).
In particular, we focus on  $v$ such  that $\S(v)$ is a bispecial factor of the Sturmian sequence $\uu = \S({\vv})$ and $\S(v)$ belongs to the set   $\mathcal{B}_\uu(n)$  described  in   Theorem  \ref{Thm_P01E}.

In fact, the following lemmas explain that only the bispecial factor $\S(v) = BS(M+a_{N+1}+1)$, where $ N \in \N, M = a_0 + a_1 + \cdots a_N$, is important. By Proposition \ref{ParikhRSB}, its return words $r$ and $s$ have either the lengths
$ |r| = q_{N+1}'$  and $|s| = q_{N+1}' + q_{N}'$   if $a_{N+2}>1 $,
or $ |r| =q_{N+1}' + q_{N}' $  and $|s| = q_{N+1}'$ if $a_{N+2}=1 $.

\begin{lem}\label{KratkeBS1}
Let $\vv$ be a CS Rote sequence and $\uu = \S(\vv)$ be the associated standard Sturmian sequence with the  directive sequence $G^{a_1}D^{a_2}G^{a_3}D^{a_4} \cdots $ or $D^{a_1}G^{a_2}D^{a_3}G^{a_4} \cdots $.
Put $M= a_0+a_1+a_2 +\cdots + a_N$, where $a_0=0$.
Let $x$ and $y$ be the bispecial factors in $\vv$ such that $\S(x) = BS(M+a_{N+1}+1)$ and $\S(y) = BS(M+m)$, where $m \in \{0,1,\ldots, a_{N+1} - 1\}$.
Then at least one return word to $x$ in $\vv$ is longer than every return word to $y$ in $\vv$.
\end{lem}

\begin{proof}
On one hand, by Lemmas \ref{Lem_ImageG}, \ref{Lem_ImageD} and Remark \ref{Rem_DerivedSeq}, the derived sequence $\dd_{\uu}(\S(y))$  is a standard Sturmian sequence with the directive sequence $G^{a_{N+1}-m}D^{a_{N+2}}G^{a_{N+3}} \cdots $ or $D^{a_{N+1}-m}G^{a_{N+2}}D^{a_{N+3}} \cdots$. It implies that $\uu$ is a concatenation of the blocks ${r'}^\ell s'$  and ${r'}^{\ell+1} s'$, where $\ell = a_{N+1}-m$.  The return words to $\S(y)$ are by Proposition \ref{ParikhRSB} of length  $|r'| =p_N+ q_N= q_N'$ and $|s'| =  m\, q_{N}'+q_{N-1}'$. Regardless of (un)stability of the return words to $\S(y)$, the longest return word to $y$ in $\vv$ is by Theorem \ref{Thm_RetWordsRote} of length at most
$$(\ell +1)|r'| + 2|s'|  =
(a_{N+1}-m +1) q_{N}' + 2(m\, q_{N}'+q_{N-1}') = q_{N+1}' +(m+1)q_{N}'
+ q_{N-1}'\leq 2q_{N+1}'.$$
On the other hand,  the return words $r, s$ to $\S(x)$ in $\uu$  have  by  Proposition \ref{ParikhRSB} either  lengths $ |r| = q_{N+1}'$  and $|s| = q_{N+1}' + q_{N}'$   (if $a_{N+2}>1 $),  or  lengths   $ |r| =q_{N+1}' + q_{N}' $  and $|s| = q_{N+1}'$ (if $a_{N+2}=1 $). Regardless of (un)stability of the return words to $\S(x)$,  one of the return words to $x$ in $\vv$ is of length at least
$$|r|+|s| =   2 q_{N+1}' + q_{N}'\,. $$
\end{proof}

\begin{lem}\label{KratkeBS2}
Let $\vv$ be a CS Rote sequence and $\uu = \S(\vv)$ be  the associated standard Sturmian sequence with the  directive sequence $G^{a_1}D^{a_2}G^{a_3}D^{a_4} \cdots $ or $D^{a_1}G^{a_2}D^{a_3}G^{a_4} \cdots $.
Put $M=a_0+ a_1+a_2 +\cdots + a_N$, where $a_0=0$.
Let $x$ and $y$ be the bispecial factors in $\vv$ such that $\S(x) = BS(M+a_{N+1}+1)$ and $\S(y) = BS(M+a_{N+1})$.  Then at least one return word to $x$ in $\vv$ is  at least as long as every  return word to $y$ in $\vv$.
\end{lem}

\begin{proof}
Let us denote the return words to $\S(y)$ by $r'$  and $s'$, and the return words to $\S(x)$ by $r$ and $s$.
By Proposition \ref{ParikhRSB}, $|r'| = q_{N+1}'$ and $|s'| = q_{N}'$.

First, we assume that $r'$  is unstable.
Then by Theorem  \ref{Thm_RetWordsRote}, the return words to $y$ in $\vv$  are of length at most $2|r'| +|s'| = 2q_{N+1}'+ q_{N}'$. Regardless of (un)stability of the return words to $\S(x)$,  one of the return words to $x$ in $\vv$ is of length at least
$|r|+|s| =   2 q_{N+1}' + q_{N}'\,. $

It remains to discuss the case when  $r'$  is stable.
Let us assume that $\uu$ has the directive sequence $G^{a_1}D^{a_2}G^{a_3}D^{a_4} \cdots$.
By Proposition \ref{ParikhRSB}, it means that $|r'|_1 = q_{N+1}$ is even.
We use Theorem \ref{Thm_RetWordsRote} to find the longest return word to $y$. Similarly as in the proof of Lemma \ref{KratkeBS1}, we determine that $\ell = a_{N+2}$ and  thus the longest return word  to $y$ in $\vv$  is of length
$$L'=2|s'| + (\ell +1) |r'|=2\, q_{N}' + (a_{N+2} +1)q_{N+1}' = q_{N+2}' + q_{N+1}' + q_{N}'\,.$$

Let us compare $L'$ with the length of the return words to $x$ in $\vv$. If $a_{N+2}>1$, then by  Proposition \ref{ParikhRSB},  $|r|_1 =  |r'|_1 = q_{N+1}$ and  $r$ is stable as well.
The longest return word to $x$ is by Theorem \ref{Thm_RetWordsRote} the return word $sr^{\ell+1}s$, where $\ell = a_{N+2}-1$.
Its length is
$$L= 2|s| + a_{N+2}|r| = 2(q_{N+1}'+ q_N') + a_{N+2}q_{N+1}'=q_{N+2}'+2q_{N+1}'+ q_N' > L'\,.$$

If $a_{N+2}=1$, then by  Proposition \ref{ParikhRSB},  $|r|_1 = q_{N+2}=q_{N+1} + q_{N}$ and $|s|_1 = q_{N+1}$.
Since $q_{N+1}$ is even, necessarily  $q_{N}$ is odd (as follows from the well-known relation $p_Nq_{N+1} -p_{N+1}q_{N} = (-1)^{N+1}$ for all $N$). It means that $r$ is unstable and $s$ is stable. Thus the longest return word to $x$ in $\vv$ has the length
$$L =2|r| + |s| = 2 (q_{N+1}'+ q_N') +q_{N+1}' = 2q_{N+2}' + q_{N+1}' > L'. $$

If $\uu$ has the directive sequence $D^{a_1}G^{a_2}D^{a_3}G^{a_4} \cdots $, the proof is analogous, we only have to take into account that the coordinates of the Parikh vectors are exchanged (see Remark~\ref{Rem_GandDexchanged}). In particular, instead of considering $q_N$ when determining the number of ones, we consider $p_N$.
\end{proof}

\begin{prop}\label{Coro_RecFunctionRote}
Let $\vv$ be a CS Rote sequence. Let $\uu$ be the standard Sturmian sequence such that $\mathcal{L(\S(\vv))} = \mathcal{L}(\uu)$ and let $\uu$ have the directive sequence $G^{a_1}D^{a_2}G^{a_3}D^{a_4} \cdots $ or $D^{a_1}G^{a_2}D^{a_3}G^{a_4} \cdots $.
Put  $M= a_0+a_1+a_2 +\cdots + a_N$, where $a_0=0$.
Let $L$ be the length of  the longest return word to the bispecial factor $v$ of $\vv$  such that $\S(v)$ is the bispecial factor  $ BS(M+a_{N+1}+1)$ in $\uu$.  Then the recurrence function of $\vv$ satisfies $R_\vv(n+1)=L+n$ for any  $n \in [q_N', q_{N+1}')$, where $N \in \mathbb N$.
\end{prop}

\begin{proof}
Let $\vv'$ be a CS Rote sequence associated with the standard Sturmian sequence $\uu$ such that $\mathcal{L}(\S(\vv)) = \mathcal{L}(\uu)$.  Clearly, $\mathcal{L}(\vv) = \mathcal{L}(\vv')$.
Since the recurrence function depends only on the language of the sequence and not on the sequence itself, we can work with the CS Rote sequence $\vv'$ instead of $\vv$.

It follows from Lemma \ref{Lem_BN} that $R_\vv(n+1) = R_{\vv'}(n+1)=L+n$, where $L$ is the length of the longest return word to a bispecial factor from the set $\mathcal{B}_{\vv'}(n+1)$.
Lemma \ref{Lem_ShortestBSinRoteSturm} shows the correspondence between the bispecial factors from the set $\mathcal{B}_{\vv'}(n+1)$ and the bispecial factors from the set $\mathcal{B}_\uu(n)$. In particular, if $v \in \mathcal{B}_{\vv'}(n+1)$, then $\S(v) \in \mathcal{B}_{\uu}(n)$.
For every $n \in [q_N', q_{N+1}')$, where $N \in \N$, the set $\mathcal{B}_\uu(n)$ is described in Theorem~\ref{Thm_P01E}.
Together with Lemmas \ref{KratkeBS1} and \ref{KratkeBS2}, it implies that the bispecial factor $v$ such that $\S(v)$ equals $BS(M+a_{N+1}+1)$ has the longest return word among all bispecial factors from the set $\mathcal{B}_{\vv'}(n+1)$.
\end{proof}

\begin{thm}\label{Thm_RecFunctionRoteDetailed}
Let $\vv$ be a CS Rote sequence and let $\uu$ be the standard Sturmian sequence such that $\mathcal{L}(\S(\vv)) = \mathcal{L}(\uu)$.
If $\uu$ has the directive sequence $G^{a_1}D^{a_2}G^{a_3}D^{a_4} \cdots $, then the value of the recurrence function $R_\vv$ for $n \in [q_N', q_{N+1}')$, $N \in \mathbb N$, is given by
\medskip

\begin{description}
\item[Case $q_N$ even] \quad  $  R_\vv(n+1) = \left\{\begin{array}{ll}  2q_{N+1}'+q_N'+n & \quad \text{if  \ \ } a_{N+2}>1,\\  2q_{N+2}'+n & \quad \text{if  \ \ } a_{N+2}=1. \end{array}\right. $
\medskip

\item[Case $q_{N+1}$ even] \quad  $  R_\vv(n+1) = \left\{\begin{array}{ll}  q_{N+2}'+2q_{N+1}'+q_N'+n & \quad \text{if  \ \ } a_{N+2}>1\\  2q_{N+2}'+q_{N+1}'+n & \quad \text{if  \ \ } a_{N+2}=1. \end{array}\right.$
\medskip

\item[Case $q_N, q_{N+1}$ odd] \quad  $  R_\vv(n+1) = \left\{\begin{array}{ll}  3q_{N+1}'+q_{N}'+n & \quad \text{if  \ \ } a_{N+2}>1\\  q_{N+3}'+q_{N+2}'+q_{N+1}'+n & \quad \text{if  \ \ } a_{N+2}=1. \end{array}\right.$
\end{description}

If $\uu$ has the directive sequence $D^{a_1}G^{a_2}D^{a_3}G^{a_4} \cdots $, then the value of the recurrence function $R_\vv$ for  $n \in [q_N', q_{N+1}')$, $N \in \mathbb N$, is given by
\medskip

\begin{description}
\item[Case $p_N$ even] \quad  $  R_\vv(n+1) = \left\{\begin{array}{ll}  2q_{N+1}'+q_N'+n & \quad \text{if  \ \ } a_{N+2}>1,\\  2q_{N+2}'+n & \quad \text{if  \ \ } a_{N+2}=1. \end{array}\right. $
\medskip

\item[Case $p_{N+1}$ even] \quad  $  R_\vv(n+1) = \left\{\begin{array}{ll}  q_{N+2}'+2q_{N+1}'+q_N'+n & \quad \text{if  \ \ } a_{N+2}>1\\  2q_{N+2}'+q_{N+1}'+n & \quad \text{if  \ \ } a_{N+2}=1. \end{array}\right.$
\medskip

\item[Case $p_N, p_{N+1}$ odd] \quad  $  R_\vv(n+1) = \left\{\begin{array}{ll}  3q_{N+1}'+q_{N}'+n & \quad \text{if  \ \ } a_{N+2}>1\\  q_{N+3}'+q_{N+2}'+q_{N+1}'+n & \quad \text{if  \ \ } a_{N+2}=1. \end{array}\right.$
\end{description}
\end{thm}

\begin{proof}
By  Proposition \ref{Coro_RecFunctionRote}, $R_\vv(n+1) = L+n$, where  $L$  is the length of the longest  return word  to the bispecial factor $v$ in $\vv$   such that $\S(v) = BS(M+a_{N+1}+1)$.
Consider first that $\uu$ has the directive sequence $G^{a_1}D^{a_2}G^{a_3}D^{a_4} \cdots $.
By Proposition  \ref{ParikhRSB}, the Parikh vectors of the return words $r$ and $s$  to the bispecial factor  $BS(M+a_{N+1}+1)$ are

\begin{enumerate}
\item
$\vec{V}(r)  = \left(\!\!\begin{array}{c}p_{N+1}\\ q_{N+1}\end{array}\!\!\right) , \ \ \vec{V}(s)  = \left(\!\!\begin{array}{c} p_{N+1}+p_{N}\\ q_{N+1}+q_{N}\end{array}\!\!\right)  \ \text{if } \  a_{N+2} >1$;

\item $\vec{V}(r)  = \left(\!\!\begin{array}{c} p_{N+1}+p_{N}\\ q_{N+1}+q_{N}\end{array}\!\!\right),    \ \ \vec{V}(s)  = \left(\!\!\begin{array}{c}p_{N+1}\\ q_{N+1}\end{array}\!\!\right)   \ \text{if } \  a_{N+2} =1$.
\end{enumerate}
Let us emphasize that at most one of the numbers $q_{N}$  and   $q_{N+1}$  is  even.  It follows from the well-known relation $p_Nq_{N+1} -p_{N+1}q_{N} = (-1)^{N+1}$ for all $N$. Moreover, let us recall that $p_N+q_N = q'_N$.
\medskip

\noindent First we discuss the case $a_{N+2}>1$.
\begin{itemize}
\item   If $q_N$ is even, then $q_{N+1}$ is odd, i.e., $r$ and $s$ are unstable. Since  $|r|< |s|$,   Item 3 of Theorem \ref{Thm_RetWordsRote}  gives
 $L = |rs| = 2q_{N+1}' + q_{N}'$.
\item   If $q_{N+1}$ is even, then $q_{N}$ is odd, i.e., $r$  is stable and $s$ is  unstable.  We use Item 1 of Theorem \ref{Thm_RetWordsRote}  with   $\ell =a_{N+2} - 1$.  Clearly, $L = 2|s| + (\ell +1)|r| = 2(q_{N+1}' + q_{N}') + a_{N+2}q_{N+1}' = q_{N+2}' + 2q_{N+1}' + q_{N}'$.
\item If both $q_{N}$,  $q_{N+1}$ are  odd, then $r$ is unstable and $s$ is stable. Item 2 of Theorem \ref{Thm_RetWordsRote}  implies
 $L = |rsr| = 3q_{N+1}' + q_{N}'$.
\end{itemize}
It remains to discuss the case $a_{N+2}=1$.
\begin{itemize}
\item   If $q_N$ is even, then $q_{N+1}$ is odd, i.e., $r$ and $s$ are unstable. Since  $|r|> |s|$,   Item 3 of Theorem \ref{Thm_RetWordsRote}  gives
 $L = |rr| = 2q_{N+1}' + 2q_{N}' = 2q_{N+2}'$.
\item   If $q_{N+1}$ is even, then $q_{N}$ is odd, i.e., $r$  is unstable and $s$ is  stable.
Item 2 of Theorem \ref{Thm_RetWordsRote}  implies
 $L = |rsr| = 3q_{N+1}' + 2q_{N}'= 2q_{N+2}' + q_{N+1}'$.
\item If both $q_{N}$,  $q_{N+1}$ are odd, then $r$ is stable and $s$ is unstable. We use Item 1 of Theorem \ref{Thm_RetWordsRote}  with   $\ell =a_{N+3}$.   Thus $L = 2|s| + (\ell +1)|r| = 2q_{N+1}' + (a_{N+3}+1)(q_{N+1}' + q_{N}') = q_{N+3}' + q_{N+2}' + q_{N+1}' $.
\end{itemize}

If $\uu$ has the directive sequence $D^{a_1}G^{a_2}D^{a_3}G^{a_4} \cdots $, then the statement of Theorem~\ref{Thm_RecFunctionRoteDetailed} will stay the same, only $q_N$ and $q_{N+1}$ will be replaced by $p_N$ and $p_{N+1}$ because the Parikh vectors of $r$ and $s$ have the coordinates exchanged, see Remark~\ref{Rem_GandDexchanged}.
\end{proof}

\begin{exam}
By Proposition~\ref{Prop_MalyExponent}, the critical exponent of the CS Rote sequence $\vv$ such that $\S(\vv)$ has the directive sequence $G(D^2 G^2)^{\omega}$ is $\CR(\vv)=2+\frac{1}{\sqrt{2}}$.
In Example~\ref{Ex_Rozdil_cr_u_Eu}, we have shown that the CS Rote sequence $\vv'$ associated to the Sturmian sequence $\S(\vv')=E(\S(\vv))$ has the critical exponent $\CR(\vv')=4+\frac{1}{1+\sqrt{2}}$.

Let us find an explicit formula for the recurrence function $R_{\vv}$, resp. $R_{\vv'}$ of the CS Rote sequence $\vv$, resp. $\vv'$.
We will see that these recurrence functions differ essentially, too.
In the proof of Proposition~\ref{Prop_MalyExponent}, we have shown that all $q_{N}$ are odd and we have found an explicit formula for $q_{N}'$,  see \eqref{reseni}.
Applying Theorem~\ref{Thm_RecFunctionRoteDetailed}, we obtain for every $n \in [q_N', q_{N+1}')$
 $$R_{\vv}(n+1)=3q_{N+1}'+q_{N}'+n=  n+ \frac{1}{2\sqrt{2}} \Bigl( (4+3\sqrt{2})(1+\sqrt{2})^{N+1} - (4-3\sqrt{2})(1-\sqrt{2})^{N+1}\Bigr)\,.   $$
Furthermore, $p_N$ is even if and only if $N$ is even.
Therefore, we obtain for every $n \in [q_{2N}', q_{2N+1}')$
 $$R_{\vv'}(n+1)=2q_{2N+1}'+q_{2N}'+n=n+\frac{1}{2\sqrt{2}}\Bigl( (1+\sqrt{2})^{2N+3} - (1-\sqrt{2})^{2N+3}\Bigr)\,;$$
 and for every $n \in [q_{2N-1}', q_{2N}')$
 $$R_{\vv'}(n+1)=q_{2N+1}'+2q_{2N}'+q_{2N-1}'+n=n+\frac{1}{\sqrt{2}}\Bigl( (1+\sqrt{2})^{2N+2} - (1-\sqrt{2})^{2N+2}\Bigr)\,.$$

\end{exam}


\nocite{*}
\bibliographystyle{abbrvnat}

\end{document}